\documentclass[11pt]{article}
\usepackage[left=1in, right=1in, top=1in, bottom=1in]{geometry}
\usepackage{graphicx}
\usepackage{amsmath,amssymb,amsfonts,amsthm}
\usepackage{mathrsfs}
\usepackage[title]{appendix}
\usepackage{xcolor}
\usepackage{booktabs}
\usepackage{float}
\usepackage[labelfont=bf]{caption}
\usepackage{authblk}
\usepackage{hyperref}
\hypersetup{colorlinks=true,
	linkcolor=blue,
	citecolor=blue,
	urlcolor=blue
}

\title{\textbf{On the cumulative residual interval entropy of doubly truncated random variables}}

\author{Stathis Chadjiconstantinidis}
\author{Apostolos Bozikas\thanks{Corresponding author. Email: \texttt{bozikas@unipi.gr}}}

\affil{Department of Statistics and Insurance Science, University of Piraeus, 80 Karaoli and Dimitriou str., Piraeus, 18534, Greece}

\date{}

\newtheorem{theorem}{Theorem}
\newtheorem{proposition}{Proposition}
\newtheorem{lemma}{Lemma}
\newtheorem{corollary}{Corollary}
\newtheorem{example}{Example}
\newtheorem{remark}{Remark}
\newtheorem{definition}{Definition}

\begin{document}

\maketitle
	
\begin{abstract}
This paper introduces and studies a new uncertainty measure, the cumulative residual interval entropy (CRIE). Defined as the cumulative residual entropy of a doubly truncated (interval) continuous random variable, this measure has several applications when data fall between two points. The CRIE generalizes the cumulative residual entropy proposed by Rao et al. \cite{Rao2004} 
and the dynamic cumulative residual entropy proposed by Asadi and Zohrevand \cite{Asadi2007}.
~We establish some properties of the generalized hazard rate and the doubly truncated mean residual lifetime, which are useful for obtaining results for the CRIE. Furthermore, we provide several representations of the CRIE based on reliability measures, covariance, the relevation transform, and increasing transformations. Finally, upper and lower bounds, as well as monotonicity results for the CRIE, are provided.
\end{abstract}
	
\vspace{1em}
	
\vspace{1em}

\section{Introduction} \label{sec1}

Entropy is a key concept in many areas, such as probability, statistics, actuarial science, communication theory, physics, and economics. In information theory, it quantifies the uncertainty of a random variable. For instance, if we are interested in studying the uncertainty of a system and/or a unit that starts its activity at time zero and continues to function regularly until it fails, then the famous Shannon’s entropy is an acceptable measure of uncertainty. 

Let $X$ be a non-negative, absolutely continuous random variable with distribution function $F_X (x)=\Pr(X\leq x)$, survival function $\overline{F}_X (x)=\Pr(X>x)=1-F_X(x)$ and probability density function $f_X (x)=F_X'(x)$. Shannon's entropy (also called as “differential entropy” in the continuous case) was introduced by \cite{Shannon1948} and defined as  
\begin{equation} \textstyle
	S(X)=-E[\ln f_X (X)]=-\int_0^\infty f_X (x)\ln f_X (x)\,dx,
\end{equation}
where “ln” denotes the natural logarithm with the convention $0\ln 0=0$. The Shannon entropy has certain disadvantages. For example, in order to estimate the Shannon entropy for a continuous density, one must obtain the density estimation, which is not a trivial task, and also, the Shannon entropy may take negative values. Rao et al. \cite{Rao2004} extended the Shannon entropy by introducing a novel measure of information known as the \textit{cumulative residual entropy} (CRE), defined by
\begin{equation} \textstyle
	\mathcal{E}(X)=-\int_0^\infty \overline{F}_X (x)\ln \overline{F}_X (x)\,dx,
\end{equation}
which is always non-negative and can easily be computed from sample data. For additional mathematical properties of $\mathcal{E}(X)$, see also Rao \cite{Rao2005}. A generalization of $\mathcal{E}(X)$ was obtained by Psarrakos and Navarro \cite{Psarrakos2013} who introduced and studied the generalized cumulative residual entropy. 

Whereas sometimes in the study of a system's lifetime, the initial age is non-zero and an individual’s interest is to study the uncertainty of the system, which is known to survive up to a certain age, or failed at a fixed time, or when it is known to fail between two time points, then $S(X)$ and $\mathcal{E}(X)$ are not suitable for uncertainty measures.  
In a situation where it is known that the system has survived up to a specific age, uncertainty is related to the residual distribution of the system as well as to a left truncation of system's lifetime. Let the left-truncated random variable $X_t=X\,|\,X>t$ having survival function  
$$ \textstyle \overline{F}_{X_t}(x)=\frac{\overline{F}_X(x)}{\overline{F}_X(t)}, \quad x>t. $$

Based on the entropy $\mathcal{E}(X)$ given by (2), Asadi and Zohrevand \cite{Asadi2007} introduced the \textit{dynamic cumulative residual entropy} at time $t$ of $X$, denoted by $\mathcal{E}(X;t)$ and defined as $\mathcal{E}(X;t)=\mathcal{E}(X_t)$, that is,  
\begin{equation} \textstyle
	\mathcal{E}(X;t)=-\int_t^\infty \frac{\overline{F}_X (x)}{\overline{F}_X (t)} \ln\left(\frac{\overline{F}_X (x)}{\overline{F}_X (t)}\right) dx.
\end{equation}
It is clear that $\mathcal{E}(X;0)=\mathcal{E}(X)$. Asadi and Zohrevand \cite{Asadi2007} obtained several properties for $\mathcal{E}(X;t)$ as well as some characterizations of lifetime distributions based on this measure. In addition, further results for $\mathcal{E}(X)$ and $\mathcal{E}(X;t)$ were obtained by \cite{Navarro2010}.  

Let $X_{(t)}=X-t\,|\,X>t$, $t\geq 0$, be the residual lifetime of $X$, having the survival function given by  
$$ \textstyle \overline{F}_{X_{(t)}} (x)=\frac{\overline{F}_X (x+t)}{\overline{F}_X (t)}, \quad x>0,
$$
provided that $\overline{F}_X(t)>0$. This random variable has important applications in several scientific fields. For example, in life insurance, actuaries utilize residual lifetime when determining life insurance premiums, mathematical reserves, and benefits. In this case, $X$ denotes the lifetime of a newborn and $X_{(t)}$ is the future lifetime of a person of age $t>0$. In non-life insurance, if $X$ represents the individual loss of an insurance portfolio and $t$ is an ordinary deductible, then $X_{(t)}$ denotes the payment-per-payment of the insurance company to the policyholder (see Klugman  et al. \cite{Klugman2019}). Ebrahimi \cite{Ebrahimi1996a} considered the differential entropy of $X_{(t)}$ which according to (1) is given by  
$$ \textstyle \mathcal{E}(X_{(t)})=-\int_0^\infty \overline{F}_{X_{(t)}}(x)\ln \overline{F}_{X_{(t)}}(x)\,dx.
$$
Since
\begin{align*}
	\mathcal{E}(X_{(t)}) &= \textstyle -\int_0^\infty \frac{\overline{F}_X(x+t)}{\overline{F}_X(t)} \ln \left(\frac{\overline{F}_X(x+t)}{\overline{F}_X(t)}\right) dx \\
	&= \textstyle -\int_t^\infty \frac{\overline{F}_X(x)}{\overline{F}_X(t)} \ln \left(\frac{\overline{F}_X(x)}{\overline{F}_X(t)}\right) dx,
\end{align*}
it follows that $\mathcal{E}(X_{(t)})=\mathcal{E}(X_t)$ (each equal to $\mathcal{E}(X;t)$), that is, the dynamic cumulative residual entropy at time $t$ of $X$ is equal to the differential entropy of the residual lifetime $X_{(t)}$ of $X$.  

In recent years, there has been growing interest in studying doubly truncated random variables in several scientific fields, such as survival studies and life testing, reliability, economy, astronomy, actuarial science, and various other fields. The subject of doubly truncation of a lifetime random variable in the reliability literature has been started by Navarro and Ruiz \cite{Navarro1996}. More recently, Kharazmi and Balakrishnan \cite{Kharazmi2021} introduced and studied the cumulative residual and relative cumulative residual Fisher information, establishing important properties and connections to measures of uncertainty under truncation, thereby further motivating the study of information measures for truncated lifetimes. If the random variable $X$ denotes the lifetime of a system, then the random variable 
$$ \textstyle X_{\tau_1,\tau_2} = X \mid \tau_1 \leq X \leq \tau_2, $$
is referred to as the \textit{doubly truncated random variable} with $(\tau_1,\tau_2)\in D_X$, where
$$ \textstyle D_X = \{(x,y) : F_X(x) < F_X(y)\}. $$
Hence, the random variable $X_{\tau_1,\tau_2}$ represents the lifetime of a system and/or a unit that was in working condition at age $\tau_1$, but out of order when observing at time $\tau_2$, that is, the lifetime of a system that fails in the interval $[\tau_1, \tau_2]$. For further results on this topic, we refer to Sankaran and Sunoj \cite{Sankaran2004} who defined and studied the expected value of doubly truncated random variables. Moreover, Sunoj et al. \cite{Sunoj2009} studied characterizations of lifetime distributions using conditional expectations of doubly truncated random variables, and Khorashadizadeh et al. \cite{Khorashadizadeh2012} investigated the doubly truncated mean residual and past lifetime.

In survival studies and life testing, significant attention has been paid to truncated lifetimes. Among the truncated study, the most commonly observed truncation is of doubly truncated situations. In many practical cases, data are available only within a limited range between two points. Therefore, it is important to analyze statistical measures, especially in fields such as information theory, reliability, actuarial science, and economics, when random variables are doubly truncated.
Hence, event times of units that lie within a specific time interval are only observed, and we cannot have access to information about the subjects outside of this interval. For example, in survival studies, one may often observe truncation situations while analyzing data with respect to time. Some instances of truncated situations include analysis of COVID-19 data on the time from infection to death, the time it takes for treatment to show a response, and the duration until the disease reoccurs.

As an example in reliability testing, a product may only be observed until a certain time point (upper truncation) after it has been in use for a given period (lower truncation), leading to doubly truncated data. A detailed understanding and proper analysis of these data are crucial to make an accurate assessment of the reliability of the system and making informed decisions about maintenance, product design, or risk management. In addition, final products are often subject to selection checks before being sent to the customer. The usual practice is that if the performance of a product falls within certain tolerance limits, it is referred as compatible and sent to the customer. If it fails, a product is rejected, and thus revoked. In this case, the actual distribution to the customer is called doubly (interval) truncated.

When studying income distribution, inequality is not measured only for incomes above or below a fixed value. It is also important to examine incomes within a given range, as well as incomes higher or lower than a specified level. For instance, it is interesting to examine the inequality of a population that excludes the richest and poorest members; hence doubly truncated populations are taken into consideration in many practical situations.

Doubly truncated random variables are widely used in actuarial science and especially are used to model insurance companies' compensation to policyholders. A “franchise deductible” (see Klugman et al. \cite{Klugman2019}) is a provision under which the insurance pays nothing when the loss is below the deductible amount, but pays the whole loss if it exceeds the deductible. Hence, if the random variable $X$ represents the individual (and / or the collective) loss of an insurance portfolio and $d > 0$ is the amount of the franchise deductible, then the random variable $Y_d^P$ representing the payment-per-payment is given by
{\small $$ Y_d^P = \begin{cases}
		\text{undefined}, & X \leq d \\
		X, & X > d
	\end{cases}, \quad \text{i.e., } Y_d^P = X \mid X > d.
	$$}

Suppose that losses are not recorded for or above a certain policy (liability) limit $u > 0$. For example, a home insurance policy covers losses up to a limit $u$, while the policyholder covers major losses that exceed the amount $u$. In this case, the random variable representing the payment of the insurance company is
{\small $$
	Z_u = \begin{cases}
		X, & X \leq u \\
		\text{undefined}, & X > u
	\end{cases}, \quad \text{i.e., } Z_u = X \mid X \leq u.
	$$}Therefore, in this case, the random variable $Z_u$ is truncated from above. It should be noted that this case is different from the loss censoring above. For the latter case, the insurance company's payment is the random variable $\min\{X, u\}$, and therefore when $X > u$, the company pays an amount equal to $u$. Now, suppose that the losses $X$ of an insurance portfolio are covered by a franchise deductible $d > 0$ and are not recorded for or above a policy limit $u > d$. Then, the payment of the insurance company to the policyholder is obviously a doubly truncated random variable $X_{d,u} = X \mid d < X \leq u$. Another example in insurance is that the claim time of a policyholder is doubly truncated between the policy’s start and maturity dates.

It can be easily shown that for all $(\tau_1, \tau_2) \in D_X$, the survival function $ \textstyle \overline{F}_{X_{\tau_1,\tau_2}}(x) = \Pr(X_{\tau_1,\tau_2} > x)$ and the probability density function $ \textstyle f_{X_{\tau_1,\tau_2}}(x) = -\frac{d}{dx} \overline{F}_{X_{\tau_1,\tau_2}}(x)$ of $X_{\tau_1,\tau_2}$ are given respectively by
$$ \textstyle 
\overline{F}_{X_{\tau_1,\tau_2}}(x) = \frac{\overline{F}_X(x) - \overline{F}_X(\tau_2)}{\overline{F}_X(\tau_1) - \overline{F}_X(\tau_2)}, \quad \tau_1 \leq x \leq \tau_2,
$$
$$ \textstyle
f_{X_{\tau_1,\tau_2}}(x) = \frac{f_X(x)}{F_X(\tau_2) - F_X(\tau_1)}, \quad \tau_1 \leq x \leq \tau_2.
$$

A measure of uncertainty for the random variable $X_{\tau_1,\tau_2}$, also known as the \textit{interval entropy} of $X$ in the interval $(\tau_1, \tau_2)$ is introduced by \cite{Sunoj2009} and is defined as 
\begin{align}
	S(X; \tau_1, \tau_2) &= \textstyle -\int_{\tau_1}^{\tau_2} f_{X_{\tau_1,\tau_2}}(x) \ln f_{X_{\tau_1,\tau_2}}(x) \, dx \notag \\
	&= \textstyle -\int_{\tau_1}^{\tau_2} \frac{f_X(x)}{F_X(\tau_2) - F_X(\tau_1)} \ln \left( \frac{f_X(x)}{F_X(\tau_2) - F_X(\tau_1)} \right) dx.
\end{align}
which is the Shannon entropy for $X_{\tau_1,\tau_2}$, i.e., $S(X;\tau_1,\tau_2) = S(X_{\tau_1,\tau_2})$. Clearly, it holds $S(X;t,\infty) = S(X;t)$ and $S(X;0,t) = \overline{S}(X;t)$. Some further properties of interval entropy were obtained by \cite{Misagh2012} and \cite{Moharana2020}. In addition, the weighted interval entropy was studied by \cite{Misagh2011}.

Based on the measures $\mathcal{E}(X)$ and $\mathcal{E}(X;t)$, in this paper we introduce the \textit{cumulative residual interval entropy} of the doubly truncated random variable $X$ in the interval $(\tau_1, \tau_2)$, denoted by $H(X; \tau_1, \tau_2)$ and defined as
$$ \textstyle H(X; \tau_1, \tau_2) = \mathcal{E}(X_{\tau_1,\tau_2}) = -\int_0^\infty \overline{F}_{X_{\tau_1,\tau_2}}(x) \ln \overline{F}_{X_{\tau_1,\tau_2}}(x) \, dx. $$
Hence, based on the definition of $\overline{F}_{X_{\tau_1,\tau_2}}(x)$ we give the following definition.

\begin{definition}
	Let $X$ be a non-negative absolutely continuous random variable. For all $(\tau_1,\tau_2)\in D_X$, the cumulative residual interval entropy (CRIE) of the doubly truncated random variable $X$ in the interval $(\tau_1,\tau_2)$, is defined as
	\begin{align}
		H(X; \tau_1, \tau_2) & = \textstyle -\int_{\tau_1}^{\tau_2} \overline{F}_{X_{\tau_1,\tau_2}}(x) \ln \overline{F}_{X_{\tau_1,\tau_2}}(x) \, dx \nonumber \\
		& = \textstyle-\int_{\tau_1}^{\tau_2} \frac{\overline{F}_X(x) - \overline{F}_X(\tau_2)}{\overline{F}_X(\tau_1) - \overline{F}_X(\tau_2)} \ln \left( \frac{\overline{F}_X(x) - \overline{F}_X(\tau_2)}{\overline{F}_X(\tau_1) - \overline{F}_X(\tau_2)} \right) dx.
	\end{align}
\end{definition}

From (5) it easily follows that $H(X; \tau_1, \tau_2)$ takes values in $[0, +\infty)$, i.e., the measure $H(X; \tau_1, \tau_2)$ is always non-negative. Also, it holds that $H(X; t, \infty) = \mathcal{E}(X; t)$ as well as $H(X; 0, \infty) = \mathcal{E}(X)$.

Let us consider the random variable $X_{(\tau_1, \tau_2)} = X - \tau_1 \mid \tau_1 \leq X \leq \tau_2$ which is the doubly truncated mean residual lifetime of $X$. Thus, $X_{(\tau_1,\tau_2)}$ represents the future lifetime of a unit and / or a system in the time interval $[\tau_1, \tau_2]$ given that it has survived to age $\tau_1$. The random variable $X_{(\tau_1,\tau_2)} \in (0, \tau_2 - \tau_1]$ and has survival function
$$ \textstyle
\overline{F}_{X_{(\tau_1,\tau_2)}}(x) = \Pr(X_{(\tau_1,\tau_2)} > x) = \frac{\overline{F}_X(x + \tau_1) - \overline{F}_X(\tau_2)}{\overline{F}_X(\tau_1) - \overline{F}_X(\tau_2)}, \quad 0 \leq x \leq \tau_2 - \tau_1. $$
Therefore, the CRE of the random variable $X_{(\tau_1,\tau_2)}$ according to (2) is
\begin{align*}
	\mathcal{E}(X_{(\tau_1,\tau_2)}) & = \textstyle -\int_0^\infty \overline{F}_{X_{(\tau_1,\tau_2)}}(x) \ln \overline{F}_{X_{(\tau_1,\tau_2)}}(x) \, dx \\
	& = \textstyle -\int_0^{\tau_2 - \tau_1} \frac{\overline{F}_X(x + \tau_1) - \overline{F}_X(\tau_2)}{\overline{F}_X(\tau_1) - \overline{F}_X(\tau_2)} \ln \left( \frac{\overline{F}_X(x + \tau_1) - \overline{F}_X(\tau_2)}{\overline{F}_X(\tau_1) - \overline{F}_X(\tau_2)} \right) dx,
\end{align*}
and since the last relationship can be equivalently rewritten as
$$ \textstyle
\mathcal{E}(X_{(\tau_1,\tau_2)}) = -\int_{\tau_1}^{\tau_2} \frac{\overline{F}_X(x) - \overline{F}_X(\tau_2)}{\overline{F}_X(\tau_1) - \overline{F}_X(\tau_2)} \ln \left( \frac{\overline{F}_X(x) - \overline{F}_X(\tau_2)}{\overline{F}_X(\tau_1) - \overline{F}_X(\tau_2)} \right) dx, $$
it follows that
$$ \textstyle \textstyle H(X;\tau_1,\tau_2 ) = \mathcal{E}(X_{(\tau_1,\tau_2)}) = \mathcal{E}(X_{\tau_1,\tau_2}). $$

\begin{remark}
	Many studies (see, e.g., \cite{Khorashadizadeh2013, Sekeh2015, Jalayeri2024}, among others) have referred to
	\( \textstyle
	\frac{\overline{F}_X(x)}{\overline{F}_X(\tau_1)-\overline{F}_X(\tau_2)}
	\)
	as the survival function of the truncated random variable \( \textstyle X_{\tau_1,\tau_2} = X \mid \tau_1 \le X \le \tau_2\), and on this basis defined the cumulative residual interval entropy as
	\[ \textstyle
	H_1(X; \tau_1, \tau_2) = -\int_{\tau_1}^{\tau_2}
	\frac{\overline{F}_X(x)}{\overline{F}_X(\tau_1)-\overline{F}_X(\tau_2)}
	\ln\!\left(
	\frac{\overline{F}_X(x)}{\overline{F}_X(\tau_1)-\overline{F}_X(\tau_2)}
	\right)\, dx.
	\]
	However, this interpretation is incorrect, as the expression
	\( \textstyle
	\frac{\overline{F}_X(x)}{\overline{F}_X(\tau_1)-\overline{F}_X(\tau_2)}
	\)
	is not associated with a valid survival function of a random variable. For this reason, it is more appropriate to consider \( \textstyle H_1(X; \tau_1, \tau_2)\) as a \emph{modified cumulative residual interval entropy}, in line with Definition~1 of \cite{Hashempour2024} for the corresponding extropies.
\end{remark}

\begin{remark}
	As pointed out by Di Crescenzo and Longobardi \cite{DiCrescenzo2009}, in many real-world applications uncertainty is not necessarily related to the future but rather to the past. This observation motivated the introduction of past entropy measures, which, unlike residual entropy measures, are defined in terms of the cumulative distribution function of the underlying random variable.
	Di Crescenzo and Longobardi \cite{DiCrescenzo2009} studied the \textit{cumulative past entropy} defined by 
	\[\textstyle
	\mathscr{CE}(X) =  - \int_{0}^{\infty} F_X(x) \ln F_X(x) \, dx,
	\]
	as well as the \textit{dynamic cumulative past entropy} given by
	\[ \textstyle
	\mathscr{CE}(X;t) =  - \int_{0}^{t} \frac{F_X(x)}{F_X(t)} \ln \frac{F_X(x)}{F_X(t)} \, dx.
	\]
	Also let \( \textstyle F_{X_{\tau_1,\tau_2}}(x) = \Pr(X_{\tau_1,\tau_2} \le x)	= \frac{F_X(x) - F_X(\tau_1)}{F_X(\tau_2) - F_X(\tau_1)} \)
	be the distribution function of $ X_{\tau_1,\tau_2} = X \mid \tau_1 \le X \le \tau_2$. Based on Definition 1, we can analogously define the \textit{cumulative past interval entropy} as
	\begin{align*}
		\overline{H}(X; \tau_1, \tau_2)
		& = \textstyle -\int_{\tau_1}^{\tau_2} F_{X_{\tau_1,\tau_2}}(x) \ln F_{X_{\tau_1,\tau_2}}(x) \, dx \nonumber \\
		& = \textstyle-\int_{\tau_1}^{\tau_2} \frac{F_X(x) - F_X(\tau_2)}{F_X(\tau_1) - F_X(\tau_2)} \ln \left( \frac{F_X(x) - F_X(\tau_2)}{F_X(\tau_1) - F_X(\tau_2)} \right) \, dx,
	\end{align*}
	for all $(\tau_1,\tau_2) \in D_X$.
	Also, we can easily observe that, by setting $Y = X - \tau_1$ and $t = \tau_2 - \tau_1$, we obtain
	\begin{align*} 
		\{ X \mid \tau_1 \le X \le \tau_2 \}
		&= \{ \tau_1 + Y \mid \tau_1 \le \tau_1 + Y \le \tau_2 \} \\
		&= \tau_1 + \{ Y \mid 0 \le Y \le \tau_2 - \tau_1 \} \\
		&= \tau_1 + \{ Y \mid 0 \le Y \le t \}.
	\end{align*}
	This relation shows that the uncertainty associated with $[X \mid \tau_1 \le X \le \tau_2]$ may be expressed in terms of the uncertainty of the shifted past lifetime $[Y \mid 0 \le Y \le t]$. Since the uncertainty of $Y_{(t)} = Y \mid 0 \le Y \le t$ is characterized by its cumulative distribution function, the above representation allows us to express $\overline{H}(X; \tau_1, \tau_2) $ through the cumulative past entropy denoted by $ \mathscr{CE}(Y; t) $. Indeed, since
	\[ \textstyle
	F_{X_{\tau_1,\tau_2}}(x)
	= \Pr(X_{\tau_1,\tau_2} \le x)
	= \Pr(Y \le x - \tau_1 \mid Y \le t)
	= \Pr(Y \le x^{\star} \mid Y \le t),
	\]
	we obtain
	\begin{align*}
		\mathscr{CE}(Y; t)
		&= \textstyle - \int_{0}^{t}
		F_{Y_{(t)}}(x^{\star}) \ln F_{Y_{(t)}}(x^{\star}) \, dx^{\star} \\
		&= \textstyle - \int_{\tau_1}^{\tau_2} F_{X_{\tau_1,\tau_2}}(x) \ln F_{X_{\tau_1,\tau_2}}(x) \, dx \\
		&= \overline{H}(X; \tau_1, \tau_2),
	\end{align*}
	where $x^{\star} = x - \tau_1$.
	
	Therefore, the cumulative past interval entropy $\overline{H}(X;\tau_1,\tau_2)$ coincides with the dynamic cumulative past entropy $\mathscr{CE}(X - \tau_1; \tau_2 - \tau_1)$ of the shifted random variable. Consequently, this quantity does not require a separate study. In contrast, the proposed measure $ H(X; \tau_1, \tau_2) $ cannot be derived from any existing entropy measure and therefore constitutes a novel measure deserving separate attention.
\end{remark}

In the sequel, we recall some concepts useful in classifying life distributions in reliability theory, as well as some concepts about the stochastic orders between random variables.

\begin{definition}
	Let $X$ be an absolutely continuous and non-negative random variable. Then $X$ has:
	\begin{itemize}
		\item[(i)] an increasing (decreasing) failure rate or IFR (DFR), if $h_X (x)$ is increasing (decreasing) in $x \geq 0$;
		\item[(ii)] an increasing (decreasing) mean residual lifetime or IMRL (DMRL), if $m_X (x)$ is increasing (decreasing) in $x \geq 0$;
		\item[(iii)] the new worse (better) than used in expectation or NWUE (NBUE) if $m_X (x) \geq (\leq) m_X (0)=\mathbb{E}(X)$.
	\end{itemize}
\end{definition}

Thus, the relations between the classes given in the above definition are
\[ \text{IFR (DFR)} \Rightarrow \text{DMRL (IMRL)} \Rightarrow \text{NBUE (NWUE)}. \]

\begin{definition}
	Let $X$ and $Y$ be two absolutely continuous non-negative random variables having distribution functions $\overline{F}_X (x)$ and $\overline{F}_Y (x)$, probability density functions $f_X (x)$ and $f_Y (x)$, and hazard rate functions $h_X (x)$ and $h_Y (x)$, respectively. Then, $X$ is smaller than $Y$ in
	\begin{itemize}
		\item[(i)] the usual stochastic order, denoted by $X \leq_{st} Y$, if $ \overline{F}_X (x) \leq \overline{F}_Y (x)$ for all $x \geq 0$;
		\item[(ii)] the hazard rate order, denoted by $X \leq_{hr} Y$, if $h_Y (x) \leq h_X (x)$ or equivalently, if $\frac{\overline{F}_Y (x)}{\overline{F}_X (x)}$ is increasing in $x \geq 0$;
		\item[(iii)] the likelihood ratio order, denoted by $X \leq_{lr} Y$, if $\textstyle \frac{f_Y (x)}{f_X (x)}$ is increasing in $x \geq 0$.
	\end{itemize}
\end{definition}
For the stochastic orders given in Definition 3, it is well known that the following relations are valid.
\[ \textstyle X \leq_{lr} Y \Rightarrow X \leq_{hr} Y \Rightarrow X \leq_{st} Y. \]

This paper aims to study the entropy measure $H(X; \tau_1, \tau_2)$. In this context, Section \ref{sec2} provides several properties satisfied by the generalized failure rate introduced by Navarro et al. \cite{Navarro1996} and the doubly truncated mean residual lifetime introduced by Ruiz and Navarro \cite{Ruiz1996}, which, apart from being useful in themselves, are also needed for our subsequent analysis. In Section \ref{sec3}, we give several alternative representations for $H(X; \tau_1, \tau_2)$, which are useful for obtaining bounds for this information measure. It is shown that CRIE is related to well-known reliability measures, especially with the expectation of doubly truncated mean residual lifetime. Also, a covariance as well as a revelation transform representation of CRIE are given, and the problem of the evaluation of CRIE under increasing transformations is considered. In Section \ref{sec4}, we obtain several upper and lower bounds for CRIE, and we obtain a relationship between the CRIE of two non-negative and absolutely continuous random variables ordered in the usual likelihood ratio order and/or ordered with respect to their generalized failure rates. Finally, in Section \ref{sec5}, we introduce two new classes of distribution, namely the increasing (decreasing) CRIE or ICRIE (DCRIE) and we give several necessary and sufficient conditions under which CRIE is monotone. As a result, some further bounds for CRIE are also obtained. In addition, it is shown that the ICRIE (DCRIE) class is closed under increasing linear transformations, and that the DCRIE class is closed under increasing convex transformations.

\section{Some results on the generalized failure rate and doubly truncated mean residual lifetime} \label{sec2}

Suppose that $X$ is a non-negative and absolutely continuous random variable, which usually represents the lifetime of a unit and / or a system. One measure that is time-dependent and determines an instantaneous rate of failure is the \textit{hazard (or failure)} rate (HR). This is one of the most important concepts in reliability theory, survival analysis, actuarial science and risk analysis, defined by 
\begin{equation*}
\textstyle	h_X(t) = \frac{f_X(t)}{\overline{F}_X(t)},
\end{equation*}
provided that $\overline{F}_X(t) > 0$.  

Another measure of ageing used frequently is the \textit{mean residual lifetime} (MRL) which represents the expected residual lifetime of a unit and/or a system of a specific age. The MRL at age $t \geq 0$ is defined by
\begin{equation*} \textstyle
	m_X(t) = E(X - t \mid X > t) = \frac{1}{\overline{F}_X(t)} \int_t^\infty \overline{F}_X(x) \, dx,
\end{equation*}
for $t$ such that $\overline{F}_X(t) > 0$.  

In analogy to the hazard rate, Navarro and Ruiz \cite{Navarro1996} introduced the concept of the \textit{generalized failure rate} (GFR) function for a doubly truncated random variable $X \mid \tau_1 \leq X \leq \tau_2$ where $(\tau_1, \tau_2) \in D_X$. The GFR function is defined by the following relationship.
\begin{equation} \textstyle
	h_{X,1}(\tau_1, \tau_2) = \lim_{dt \to 0^+} \left\{ \frac{\Pr(\tau_1 \leq X \leq \tau_1 + dt \mid \tau_1 \leq X \leq \tau_2)}{dt} \right\} = \frac{f_X(\tau_1)}{\overline{F}_X(\tau_1) - \overline{F}_X(\tau_2)},
\end{equation}
for all $(\tau_1, \tau_2) \in D_X$. Note that $h_{X,1}(\tau_1, \tau_2)$ is the failure rate of the random variable $X \mid \tau_1 \leq X \leq \tau_2$ at time $\tau_1$. Also, when $\tau_2 \to \infty$, it holds that $h_{X,1}(t, \infty) = h_X(t)$. It should be noted that \cite{Navarro1996} also defined the following GFR.
\begin{equation*} \textstyle
	h_{X,2}(\tau_1, \tau_2) = \lim_{dt \to 0^-} \left\{ \frac{\Pr(\tau_2 + dt \leq X \leq \tau_2 \mid \tau_1 \leq X \leq \tau_2)}{dy} \right\} = \frac{f_X(\tau_2)}{\overline{F}_X(\tau_1) - \overline{F}_X(\tau_2)},
\end{equation*}
for all $(\tau_1, \tau_2) \in D_X$.  

Let $ \textstyle m_{X,1}(\tau_1, \tau_2) = E(X - \tau_1 \mid \tau_1 \leq X \leq \tau_2) = \frac{1}{\overline{F}_X(\tau_1) - \overline{F}_X(\tau_2)} \int_{\tau_1}^{\tau_2} (x - \tau_1) f_X(x) \, dx, $
denote the \textit{doubly truncated mean residual lifetime} of $X$ (see, e.g., \cite{Ruiz1996}). Then, it can be shown that 
\begin{equation} \textstyle
	m_{X,1}(\tau_1, \tau_2) = \frac{1}{\overline{F}_X(\tau_1) - \overline{F}_X(\tau_2)} \left\{ \int_{\tau_1}^{\tau_2} \overline{F}_X(x) \, dx - (\tau_2 - \tau_1) \overline{F}_X(\tau_2) \right\},
\end{equation}
and
\begin{equation} \textstyle
	m_{X,1}(\tau_1, \tau_2) = \int_{\tau_1}^{\tau_2} \frac{\overline{F}_X(x) - \overline{F}_X(\tau_2)}{\overline{F}_X(\tau_1) - \overline{F}_X(\tau_2)} \, dx.
\end{equation}
It can be directly verified that the following relationship between $h_{X,1}(\tau_1, \tau_2)$ and $m_{X,1}(\tau_1, \tau_2)$ holds. 
\begin{equation}
	h_{X,1}(\tau_1, \tau_2) = \frac{1 + \frac{\partial m_{X,1}(\tau_1, \tau_2)}{\partial \tau_1}}{m_{X,1}(\tau_1, \tau_2)}.
\end{equation}

In addition, Ruiz and Navarro \cite{Ruiz1996} introduced the \textit{doubly truncated mean past lifetime} of $X$, which is defined by $m_{X,2}(\tau_1, \tau_2) = E(\tau_2 - X \mid \tau_1 \leq X \leq \tau_2)$. The $m_{X,1}(\tau_1, \tau_2)$ can be interpreted as the expected additional lifetime for a unit that was functioning at an age $\tau_1$ and ceased to function before an age $\tau_2$, and $m_{X,2}(\tau_1, \tau_2)$ as the inactivity expected time for this unit.  

It is well-known that both $h_{X,1}(\tau_1, \tau_2)$ and $h_{X,2}(\tau_1, \tau_2)$ univocally determine the distribution function of $X$. This is also true for the functions $m_{X,1}(\tau_1, \tau_2)$ and $m_{X,2}(\tau_1, \tau_2)$.  

The \textit{doubly truncated mean function} for all $(\tau_1, \tau_2) \in D_X$ is defined by 
\begin{equation*} \textstyle
	\mu_X(\tau_1, \tau_2) = E\{X \mid \tau_1 \leq X \leq \tau_2\} = \frac{1}{F_X(\tau_2) - F_X(\tau_1)} \int_{\tau_1}^{\tau_2} x f_X(x) \, dx,
\end{equation*}
and from the definitions of $m_{X,1}(\tau_1, \tau_2)$ and $m_{X,2}(\tau_1, \tau_2)$ it holds that
\begin{equation}
	\mu_X(\tau_1, \tau_2) = \tau_1 + m_{X,1}(\tau_1, \tau_2) = \tau_2 - m_{X,2}(\tau_1, \tau_2).
\end{equation}

In the sequel, we give several results regarding the generalized hazard rate $h_{X,1}(t, \tau_2)$ and the doubly truncated mean residual lifetime $m_{X,1}(\tau_1, \tau_2)$, which, as stated previously, apart from being useful in themselves, are also needed for our subsequent analysis.

\begin{lemma}
	For any $(\tau_1, \tau_2) \in D_X$ it holds
	
	\noindent (i) \begin{equation} \textstyle \frac{\overline{F}_X(x) - \overline{F}_X(\tau_2)}{\overline{F}_X(\tau_1) - \overline{F}_X(\tau_2)} = \exp\left\{ -\int_{\tau_1}^x h_{X,1}(t, \tau_2) \, dt \right\}  \ \text{and,} \end{equation}
	
	\noindent (ii) \begin{equation} \textstyle m_{X,1}(\tau_1, \tau_2) = \int_{\tau_1}^{\tau_2} \exp\left\{ -\int_{\tau_1}^x h_{X,1}(t, \tau_2) \, dt \right\} \, dx. \end{equation}
\end{lemma}

\begin{proof} (i) We have  
	$$ \textstyle h_{X,1}(t, \tau_2) = -\frac{(\overline{F}_X(t) - \overline{F}_X(\tau_2))'}{\overline{F}_X(t) - \overline{F}_X(\tau_2)}, $$
	and thus  
	$$ \textstyle \int_{\tau_1}^x \frac{(\overline{F}_X(t) - \overline{F}_X(\tau_2))'}{\overline{F}_X(t) - \overline{F}_X(\tau_2)} \, dt = -\int_{\tau_1}^x h_{X,1}(t, \tau_2) \, dt, $$
	which yields  
	$$ \textstyle \ln[\overline{F}_X(x) - \overline{F}_X(\tau_2)] - \ln[\overline{F}_X(\tau_1) - \overline{F}_X(\tau_2)] = -\int_{\tau_1}^x h_{X,1}(t, \tau_2) \, dt. $$
	Therefore, we get  
	$$ \textstyle \ln\left( \frac{\overline{F}_X(x) - \overline{F}_X(\tau_2)}{\overline{F}_X(\tau_1) - \overline{F}_X(\tau_2)} \right) = -\int_{\tau_1}^x h_{X,1}(t, \tau_2) \, dt, $$
	from which we immediately obtain (11).
	
	\noindent (ii) By integrating both sides of (11) from $\tau_1$ to $\tau_2$ and using (8), we obtain (12). 
\end{proof}

\begin{lemma}
	(i) The generalized failure rate $h_{X,1}(x,y)$ is increasing (decreasing) in $x$ for any fixed value of $y$, if and only if the function
	$$ \textstyle \frac{\overline{F}_X (x+t)-\overline{F}_X (y)}{\overline{F}_X (x)-\overline{F}_X (y)}, \quad t>0 $$
	is decreasing (increasing) in $x$. \\
	(ii) If $h_{X,1}(x,y)$ increases (decreases) in $x$ for any fixed value of $y$, then $m_{X,1}(x,y)$ decreases (increases) in $x$.
\end{lemma}

\begin{proof}
	(i) From the definition of $h_{X,1}(x,y)$ we have that
	\begin{align*}
		h_{X,1}(x,y) & = \textstyle \lim_{t \to 0^+} \left\{ \frac{1}{t} \Pr(x \leq X \leq x+t \mid x \leq X \leq y) \right\} \\
		& = \textstyle \lim_{t \to 0^+} \left\{ \frac{1}{t} \cdot \frac{\Pr(x \leq X \leq x+t)}{\Pr(x \leq X \leq y)} \right\} \\
		& = \textstyle \lim_{t \to 0^+} \left\{ \frac{1}{t} \cdot \frac{\Pr(x \leq X \leq y)-\Pr(x+t \leq X \leq y)}{\Pr(x \leq X \leq y)} \right\} \\
		& = \textstyle \lim_{t \to 0^+} \left\{ \frac{1}{t} \left(1 - \frac{\Pr(x+t \leq X \leq y)}{\Pr(x \leq X \leq y)} \right) \right\},
	\end{align*} 
	and hence $h_{X,1}(x,y)$ is increasing (decreasing) in $x$ if and only if the function 
	$$ \textstyle \frac{\Pr(x+t \leq X \leq y)}{\Pr(x \leq X \leq y)}
	$$
	is decreasing (increasing) in $x$. Thus, the proof is completed.
	
	\noindent (ii) From (8) we get
	$$ \textstyle m_{X,1}(x,y) = \int_x^{\tau_2} \frac{\overline{F}_X(z)-\overline{F}_X(y)}{\overline{F}_X(x)-\overline{F}_X(y)} dz = \int_0^{\tau_2 - x} \frac{\overline{F}_X(x+t)-\overline{F}_X(y)}{\overline{F}_X(x)-\overline{F}_X(y)} dt, $$
	and thus
	\begin{align*}
		\frac{\partial m_{X,1}(x,y)}{\partial x} & = \textstyle -\frac{\overline{F}_X(x+y-x)-\overline{F}_X(y)}{\overline{F}_X(x)-\overline{F}_X(y)} + \int_0^{\tau_2 - x} \frac{\partial}{\partial x} \left( \frac{\overline{F}_X(x+t)-\overline{F}_X(y)}{\overline{F}_X(x)-\overline{F}_X(y)} \right) dt \\
		& = \textstyle \int_0^{\tau_2 - x} \frac{\partial}{\partial x} \left( \frac{\overline{F}_X(x+t)-\overline{F}_X(y)}{\overline{F}_X(x)-\overline{F}_X(y)} \right) dt.
	\end{align*}
	
	Therefore, if $h_{X,1}(x,y)$ is increasing (decreasing) in $x$, from (i) it follows that $ \frac{\partial m_{X,1}(x,y)}{\partial x} \leq (\geq) 0 $ and hence the result follows.
\end{proof}

\begin{lemma}
	Let $X$ and $Y$ be two non-negative absolutely continuous random variables with finite generalized failure rates $h_{X,1}(\tau_1,\tau_2)$ and $h_{Y,1}(\tau_1,\tau_2)$, respectively. If for any $(\tau_1,\tau_2)\in D_X\cap D_X$ it holds $h_{X,1}(\tau_1,\tau_2)\leq (\geq) h_{Y,1}(\tau_1,\tau_2)$, then for any increasing (decreasing) function $\varphi$
	\begin{equation}
		E[\varphi(Y) \mid \tau_1 < Y \leq \tau_2 ] \leq (\geq) E[\varphi(X) \mid \tau_1 \leq X \leq \tau_2].
	\end{equation}
\end{lemma}

\begin{proof}
	We have
	{\small \begin{align*}
			& \textstyle E[\varphi(X) \mid \tau_1 \leq X \leq \tau_2] \\
			& = \textstyle \frac{1}{\overline{F}_X(\tau_1) - \overline{F}_X(\tau_2)} \int_{\tau_1}^{\tau_2} \varphi(x) f_X(x) dx \\
			& = \textstyle-\frac{1}{\overline{F}_X(\tau_1) - \overline{F}_X(\tau_2)} \int_{\tau_1}^{\tau_2} \varphi(x) \overline{F}_X'(x) dx \\
			& = \textstyle \frac{1}{\overline{F}_X(\tau_1) - \overline{F}_X(\tau_2)} \left\{ \varphi(\tau_1)\overline{F}_X(\tau_1) - \varphi(\tau_2)\overline{F}_X(\tau_2) + \int_{\tau_1}^{\tau_2} \varphi'(x)\overline{F}_X(x) dx \right\} \\
			& = \textstyle \varphi(\tau_1) + \frac{1}{\overline{F}_X(\tau_1) - \overline{F}_X(\tau_2)} \left\{ \int_{\tau_1}^{\tau_2} \varphi'(x)\overline{F}_X(x) dx - [\varphi(\tau_2) - \varphi(\tau_1)]\overline{F}_X(\tau_2) \right\} \\
			& = \textstyle \varphi(\tau_1) + \int_{\tau_1}^{\tau_2} \varphi'(x) \frac{\overline{F}_X(x) - \overline{F}_X(\tau_2)}{\overline{F}_X(\tau_1) - \overline{F}_X(\tau_2)} dx.
	\end{align*}}
	Using (11), the last relationship implies that
	\begin{equation}
		\textstyle E[\varphi(X) \mid \tau_1 \leq X \leq \tau_2] = \varphi(\tau_1) + \int_{\tau_1}^{\tau_2} \varphi'(x) \exp\left\{-\int_{\tau_1}^x h_{X,1}(t,\tau_2) dt \right\} dx.
	\end{equation}
	If $ \textstyle h_{X,1}(\tau_1,\tau_2)\leq (\geq) h_{Y,1}(\tau_1,\tau_2)$ and since $ \textstyle \varphi'(x)\geq (\leq) 0$ if $\varphi(x)$ is increasing (decreasing) in $x\geq 0$, from (14) it follows that
	\begin{align*}
		E[\varphi(X) \mid \tau_1 \leq X \leq \tau_2] & \textstyle \geq (\leq) \varphi(\tau_1) + \int_{\tau_1}^{\tau_2} \varphi'(x) \exp\left\{-\int_{\tau_1}^x h_{Y,1}(t,\tau_2) dt \right\} dx \\
		& = \textstyle E[\varphi(Y) \mid \tau_1 < Y \leq \tau_2],
	\end{align*}
	where the last equality follows again from (14). Hence, (13) was proved.
\end{proof}

\begin{remark} (i) From Navarro et al. \cite[Proposition 2.1]{Navarro1997}, we have that it holds $h_{X,1}(\tau_1,\tau_2)\leq h_{Y,1}(\tau_1,\tau_2)$ if and only if $Y \leq_{lr} X$.
	
	\noindent (ii) Here we shall give another proof of (13) when the function $\varphi(x)$ is increasing. For this we need to recall the following two stochastic orders.
	\begin{itemize}
		\item[$\bullet$] The random variable $X$ is smaller than $Y$ in the increasing convex order (denoted by $X \leq_{icx} Y$) if (see \cite{Shaked2007}).
		$$ \textstyle E[\varphi(X)] \leq E[\varphi(Y)], \quad \text{for all increasing convex functions } \varphi. $$
		
		\item[$\bullet$] The random variable $X$ is said to be smaller than $Y$ in the mean doubly truncated order (denoted as $X \leq_{mdt} Y$) if $\mu_X(x,y) \leq \mu_Y(x,y)$ for all $(x,y)\in D_X \cap D_Y$ (see \cite{Navarro1997}).
	\end{itemize}
	From Navarro et al. \cite[Proposition 3.3]{Navarro1997}, we get that if $Y \leq_{lr} X$ then $Y \leq_{mdt} X$, which in turn implies that (see \cite[Proposition 3.1]{Navarro1997})
	$$ \textstyle [Y \mid \tau_1 < Y \leq \tau_2] \leq_{icx} [X \mid \tau_1 \leq X \leq \tau_2]. $$
	Hence it follows that if $Y \leq_{lr} X$ (which is equivalent to $h_{X,1}(\tau_1,\tau_2) \leq h_{Y,1}(\tau_1,\tau_2)$), then (13) holds.
\end{remark}

For any non-negative absolutely continuous random variables $X$ and $Y$ consider the function
\begin{equation} \textstyle
	g_{X,Y} (x;\tau_1,\tau_2) = \frac{\left(\frac{\overline{F}_Y (x) - \overline{F}_Y (\tau_2)}{\overline{F}_Y (\tau_1) - \overline{F}_Y (\tau_2)}\right)}{\left(\frac{\overline{F}_X (x) - \overline{F}_X (\tau_2)}{\overline{F}_X (\tau_1) - \overline{F}_X (\tau_2)}\right)} = \frac{\overline{F}_{Y_{\tau_1,\tau_2}} (x)}{\overline{F}_{X_{\tau_1,\tau_2}} (x)}.
\end{equation}

\begin{lemma}
	Let $X$ and $Y$ be two non-negative absolutely continuous random variables with finite generalized failure rates $h_{X,1} (\tau_1,\tau_2)$ and $h_{Y,1} (\tau_1,\tau_2)$, respectively. We have $ h_{Y,1} (\tau_1,\tau_2) \leq (\geq) h_{X,1} (\tau_1,\tau_2) $ if and only if $g_{X,Y} (x;\tau_1,\tau_2)$ is increasing in $x$.
\end{lemma}
\begin{proof}
	Using (11) we obtain that
	$$ \textstyle g_{X,Y} (x;\tau_1,\tau_2) = \exp\left\{ \int_{\tau_1}^x \left[ h_{X,1} (t,\tau_2) - h_{Y,1} (t,\tau_2) \right] dt \right\}, $$
	which implies that
	$$ \textstyle \frac{\partial g_{X,Y} (x;\tau_1,\tau_2)}{\partial x} = \left[ h_{X,1} (t,\tau_2) - h_{Y,1} (t,\tau_2) \right] g_{X,Y} (x;\tau_1,\tau_2), $$
	from which the result follows since $g_{X,Y} (x;\tau_1,\tau_2) \geq 0$.
\end{proof}

From Remark 3 (i) and Lemma 4 it follows that $X \leq_{lr} Y$ if and only if the function $g_{X,Y} (x;\tau_1,\tau_2)$ is increasing in $x$.

Let $X_e$ denote the equilibrium random variable corresponding to $X$. If $f_{X_e}$ and $\overline{F}_{X_e}$ denote respectively the probability density function and the survival function of $X_e$, then provided that $E(X)<\infty$,
$$ \textstyle f_{X_e} (x) = \frac{\overline{F}_X (x)}{E(X)} \quad \text{and} \quad \overline{F}_{X_e} (x) = \frac{\int_x^\infty \overline{F}_X (t) dt}{E(X)}. $$
Let also $ \textstyle h_{X_e,1} (x,y) = \frac{f_{X_e} (x)}{\overline{F}_{X_e} (x) - \overline{F}_{X_e} (y)} $
denote the GFR of $X_e$.

\begin{lemma}
	If the generalized failure rate $h_{X,1} (x,y)$ is decreasing (increasing) in $x$ for any fixed value of $y$, then
	$$ h_{X_e,1} (x,y) \leq (\geq) h_{X,1} (x,y). $$
\end{lemma}

\begin{proof} From (15), we get that
	\begin{align*} 
		g_{X,X_e} (x;\tau_1,\tau_2) & \textstyle = \frac{\overline{F}_X (\tau_1) - \overline{F}_X (\tau_2)}{\overline{F}_{X_e} (\tau_1) - \overline{F}_{X_e} (\tau_2)} \cdot \frac{\overline{F}_{X_e} (x) - \overline{F}_{X_e} (\tau_2)}{\overline{F}_X (x) - \overline{F}_X (\tau_2)} \\
		& \textstyle = \frac{\overline{F}_X (\tau_1) - \overline{F}_X (\tau_2)}{\overline{F}_{X_e} (\tau_1) - \overline{F}_{X_e} (\tau_2)} \cdot \frac{\int_x^{\tau_2} f_{X_e} (t) dt}{\int_x^{\tau_2} f_X (t) dt},
	\end{align*}
	and
	$$ \textstyle \frac{\partial}{\partial x} \left( \frac{\int_x^{\tau_2} f_{X_e} (t) dt}{\int_x^{\tau_2} f_X (t) dt} \right) = \frac{1}{\left(\int_x^{\tau_2} f_X (t) dt\right)^2 E(X)} \int_x^{\tau_2} \left[ f_X (x) \overline{F}_X (t) - f_X (t) \overline{F}_X (x) \right] dt. $$
	If $h_{X,1} (x,y)$ is decreasing in $x$, it also follows that the failure rate $h_X (x)$ is decreasing in $x$, and thus for $t \geq x$ it holds $h_X (x) \geq h_X (t)$ implying that
	$ f_X (x) \overline{F}_X (t) \geq f_X (t) \overline{F}_X (x) $. Therefore, it holds that $g_{X,X_e} (x;\tau_1,\tau_2)$ is increasing in $x$, and hence the result follows from Lemma 4. Similarly, if $h_{X,1} (x,y)$ is increasing in $x$, it can be shown that $g_{X_e,X} (x;\tau_1,\tau_2)$ is increasing in $x$ and thus from Lemma 4 we obtain $ h_{X,1} (x,y) \leq h_{X_e,1} (x,y). $
\end{proof}

\section{Representations for $\boldsymbol{H(X;\tau_1,\tau_2)}$}  \label{sec3}

From (5) and using (8), an alternative representation of $H(X;\tau_1,\tau_2)$ is given by
\begin{align}
	\textstyle H(X;\tau_1,\tau_2) = m_{X,1}(\tau_1,\tau_2) \ln[\overline{F}_X(\tau_1)-\overline{F}_X(\tau_2)] \textstyle - \int_{\tau_1}^{\tau_2} \frac{\overline{F}_X(x)-\overline{F}_X(\tau_2)}{\overline{F}_X(\tau_1)-\overline{F}_X(\tau_2)} \ln[\overline{F}_X(x)-\overline{F}_X(\tau_2)] \, dx.
\end{align}

\begin{example}
	Table \ref{tab:H} provides the values of $H(X; \tau_1, \tau_2)$ based on the survival functions $\overline{F}(x)$ of the following distributions: The power distribution (PD) with $\overline{F}(x) = 1 - (x/b)^a$ for $0 < x < b$, $a > 0$ and $b \neq 1$, the beta distribution with $\overline{F}(x) = 1 - x^c$ for $0 < x < 1$ and $c > 0$, the exponential distribution with $\overline{F}(x) = \exp(-\lambda x)$ for $x \geq 0$ and $\lambda > 0$, and the Lomax (or Pareto type II) distribution with $\overline{F}(x) = \left[\lambda/(\lambda + x)\right]^\alpha$ for $x \geq 0$ and $\lambda, \alpha > 0$. In Table \ref{tab:H}, we can easily observe that $H(X; \tau_1, \tau_2)$ is decreasing in $\tau_1$ for fixed values of $\tau_2$, and increasing in $\tau_2$ for fixed values of $\tau_1$.

\begin{table}[H]
\centering
\caption{Values of $H(X; \tau_1, \tau_2)$ for various distributions.}\label{tab:H}
\begin{tabular}{@{}cccccc@{}}
\toprule
$\tau_1$ & $\tau_2$ & PD(0.1, 0.9) & PD(0.3, 0.9) & Beta(0.2,1) & Beta(0.5,1) \\
\midrule
0.1 & 0.6 & 0.13182 & 0.13180 & 0.13193 & 0.13084  \\ 
0.3 & 0.6 & 0.07722 & 0.07687 & 0.07705 & 0.07643  \\ 
0.5 & 0.6 & 0.02522 & 0.02517 & 0.02520 & 0.02513  \\ 
0.1 & 0.7 & 0.15839 & 0.15862 & 0.15868 & 0.15752  \\ 
0.3 & 0.7 & 0.10347 & 0.10298 & 0.10325 & 0.10232  \\ 
0.5 & 0.7 & 0.05078 & 0.05063 & 0.05071 & 0.05047  \\ 
0.1 & 0.9 & 0.21145 & 0.21241 & 0.21223 & 0.21111  \\ 
0.3 & 0.9 & 0.15629 & 0.15559 & 0.15599 & 0.15447  \\ 
0.5 & 0.9 & 0.10258 & 0.10214 & 0.10237 & 0.10162  \\ 
\midrule
$\tau_1$ & $\tau_2$ & Exp(0.5) & Exp(1) & Lomax(2,1) & Lomax(3,1)\\ 
\midrule
3 & 10 & 1.52470 & 0.97614 & 1.61098 & 1.45535 \\ 
7 & 10 & 0.76254 & 0.68870 & 0.77343 & 0.77164 \\ 
9 & 10 & 0.25514 & 0.25652 & 0.25364 & 0.25439 \\ 
3 & 12 & 1.73694 & 0.99480 & 1.95250 & 1.69662 \\ 
7 & 12 & 1.20018 & 0.90432 & 1.28394 & 1.26233 \\ 
9 & 12 & 0.76254 & 0.68870 & 0.77200 & 0.77252 \\ 
3 & 15 & 1.90245 & 0.99955 & 2.38429 & 1.96523 \\ 
7 & 15 & 1.64355 & 0.98871 & 2.00701 & 1.91438 \\ 
9 & 15 & 1.37740 & 0.95123 & 1.54259 & 1.51885 \\ 
\bottomrule
\end{tabular}
\end{table}

\end{example}

The following theorem shows the relation between the entropy $H(X;\tau_1,\tau_2)$ and the doubly truncated mean residual lifetime $m_{X,1}(\tau_1,\tau_2)$ of $X$.

\begin{theorem} Let $X$ be a non-negative absolutely continuous random variable. Then, for all $(\tau_1,\tau_2) \in D_X$
	\begin{equation}
		H(X;\tau_1,\tau_2) = \mathbb{E}[m_{X,1}(X,\tau_2) \mid \tau_1 \leq X \leq \tau_2].
	\end{equation}
\end{theorem}

\begin{proof} It holds
	$$ \textstyle \mathbb{E}[m_{X,1}(X,\tau_2) \mid \tau_1 \leq X \leq \tau_2] = \frac{1}{\overline{F}_X(\tau_1) - \overline{F}_X(\tau_2)} \int_{\tau_1}^{\tau_2} m_{X,1}(t,\tau_2) f_X(t) dt, $$
	and hence using (8) we get
	\begin{align*}
		\mathbb{E}[m_{X,1}(X,\tau_2) \mid \tau_1 \leq X \leq \tau_2] & \textstyle = \frac{1}{\overline{F}_X(\tau_1) - \overline{F}_X(\tau_2)} \\
		& \textstyle \times \int_{\tau_1}^{\tau_2} \frac{1}{\overline{F}_X(t)-\overline{F}_X(\tau_2)} \int_t^{\tau_2} [\overline{F}_X(x)-\overline{F}_X(\tau_2)] dx \, f_X(t) dt.
	\end{align*} 
	Applying Fubini’s Theorem, the last relation yields
	{\small \begin{align*}
			& \textstyle \qquad \mathbb{E}[m_{X,1}(X,\tau_2) \mid \tau_1 \leq X \leq \tau_2] \\
			& \textstyle \qquad \qquad = \frac{1}{\overline{F}_X(\tau_1) - \overline{F}_X(\tau_2)} \int_{\tau_1}^{\tau_2} [\overline{F}_X(x)-\overline{F}_X(\tau_2)] \left( \int_{\tau_1}^x \frac{f_X(t)}{\overline{F}_X(t)-\overline{F}_X(\tau_2)} dt \right) dx \\
			& \textstyle \qquad \qquad = - \frac{1}{\overline{F}_X(\tau_1) - \overline{F}_X(\tau_2)} \int_{\tau_1}^{\tau_2} [\overline{F}_X(x)-\overline{F}_X(\tau_2)] \left( \int_{\tau_1}^x \frac{d}{dt} \ln[\overline{F}_X(t)-\overline{F}_X(\tau_2)] dt \right) dx \\
			& \textstyle \qquad \qquad = H(X;\tau_1,\tau_2).
	\end{align*}}
	
	When $\tau_2 \to \infty$, (17) is reduced to $\mathcal{E}(X;t)=\mathbb{E}[m_X(X) \mid X > t]$ which is proved in Theorem 3.1 of \cite{Asadi2007}.
\end{proof}

\begin{example} If $X$ is uniformly distributed in $[0,b]$, then $m_{X,1}(\tau_1,\tau_2) = \frac{\tau_2 - \tau_1}{2}$, and hence from Theorem 1 $$ \textstyle H(X;\tau_1,\tau_2) = \mathbb{E} \left[ \frac{\tau_2 - X}{2} \mid \tau_1 \leq X \leq \tau_2 \right] = \frac{\tau_2 - \tau_1}{4}. $$
	This shows that the CRIE for the uniform distribution is an increasing function in $\tau_2$ and a decreasing function in $\tau_1$. Hence, as $\tau_2$ gets smaller (larger), the uncertainty gets smaller (larger). Similarly, as $\tau_1$ gets larger (smaller), the uncertainty gets smaller (larger).
\end{example}

In Theorem 2 we give a covariance representation of $H(X;\tau_1,\tau_2)$, which is very useful to obtain upper bounds of this information measure.

\begin{theorem} Let $X$ be a non-negative absolutely continuous random variable. Then for all $(\tau_1,\tau_2) \in D_X$
	\begin{equation}
		\textstyle H(X;\tau_1,\tau_2) = \mathrm{Cov}\left(X, -\ln[\overline{F}_X(X)-\overline{F}_X(\tau_2)] \mid \tau_1 \leq X \leq \tau_2 \right).
	\end{equation}
\end{theorem}

\begin{proof} 
	We have
	\begin{align}
		&\mathrm{Cov}\left(X, -\ln[\overline{F}_X(X)-\overline{F}_X(\tau_2)] \mid \tau_1 \leq X \leq \tau_2 \right) \nonumber \\
		& \textstyle \qquad = -\mathbb{E}[X \ln[\overline{F}_X(X)-\overline{F}_X(\tau_2)] \mid \tau_1 \leq X \leq \tau_2] \nonumber \\
		& \textstyle \qquad \qquad + \mu_X(\tau_1,\tau_2) \mathbb{E}[\ln[\overline{F}_X(X)-\overline{F}_X(\tau_2)] \mid \tau_1 \leq X \leq \tau_2].
	\end{align}
	It holds
	\begin{align}
		& \mathbb{E}[\ln[\overline{F}_X(X)-\overline{F}_X(\tau_2)] \mid \tau_1 \leq X \leq \tau_2] \nonumber \\
		& \textstyle \qquad \qquad \qquad \qquad = \frac{1}{\overline{F}_X(\tau_1)-\overline{F}_X(\tau_2)} \int_{\tau_1}^{\tau_2} \ln[\overline{F}_X(x)-\overline{F}_X(\tau_2)] f_X(x) dx \nonumber \\
		& \textstyle \qquad \qquad \qquad \qquad = \frac{1}{\overline{F}_X(\tau_1)-\overline{F}_X(\tau_2)} \int_{\tau_1}^{\tau_2} \ln[\overline{F}_X(x)-\overline{F}_X(\tau_2)] (\overline{F}_X(x))' dx \nonumber \\
		& \textstyle \qquad \qquad \qquad \qquad = -\frac{1}{\overline{F}_X(\tau_1)-\overline{F}_X(\tau_2)} \int_{\tau_1}^{\tau_2} \ln[\overline{F}_X(x)-\overline{F}_X(\tau_2)] (\overline{F}_X(x)-\overline{F}_X(\tau_2))' dx \nonumber \\
		& \textstyle \qquad \qquad \qquad \qquad = \ln[\overline{F}_X(\tau_1)-\overline{F}_X(\tau_2)] - 1,
	\end{align}
	and similarly,
	\begin{align}
		& \qquad \mathbb{E}[X \ln[\overline{F}_X(X)-\overline{F}_X(\tau_2)] \mid \tau_1 \leq X \leq \tau_2] \nonumber \\
		& \textstyle \qquad \quad = -\frac{1}{\overline{F}_X(\tau_1)-\overline{F}_X(\tau_2)} \int_{\tau_1}^{\tau_2} x \ln[\overline{F}_X(x)-\overline{F}_X(\tau_2)] (\overline{F}_X(x)-\overline{F}_X(\tau_2))' dx \nonumber \\
		& \textstyle \qquad \quad = \tau_1 \ln[\overline{F}_X(\tau_1)-\overline{F}_X(\tau_2)] \nonumber \\
		& \textstyle \qquad \qquad + \frac{1}{\overline{F}_X(\tau_1)-\overline{F}_X(\tau_2)} \int_{\tau_1}^{\tau_2} [\overline{F}_X(x)-\overline{F}_X(\tau_2)] \ln[\overline{F}_X(x)-\overline{F}_X(\tau_2)] dx - \mu_X(\tau_1,\tau_2).
	\end{align}
	Hence, substituting (20) and (21) into the right-hand side of (19) we obtain
	\begin{align}
		& \textstyle \qquad \mathrm{Cov}\left(X, -\ln[\overline{F}_X(X)-\overline{F}_X(\tau_2)] \mid \tau_1 \leq X \leq \tau_2 \right) = -\tau_1 \ln[\overline{F}_X(\tau_1)-\overline{F}_X(\tau_2)] \nonumber \\
		& \textstyle \qquad \quad - \int_{\tau_1}^{\tau_2} \frac{\overline{F}_X(x)-\overline{F}_X(\tau_2)}{\overline{F}_X(\tau_1)-\overline{F}_X(\tau_2)} \ln[\overline{F}_X(x)-\overline{F}_X(\tau_2)] dx \nonumber \\
		& \textstyle \qquad \qquad + \mu_X(\tau_1,\tau_2) \ln[\overline{F}_X(\tau_1)-\overline{F}_X(\tau_2)]. 
	\end{align}
	Now, using (8) and (10) it follows that
	$$ \textstyle \mu_X(\tau_1, \tau_2) = \tau_1 + \int_{\tau_1}^{\tau_2} \frac{\overline{F}_X(x) - \overline{F}_X(\tau_2)}{\overline{F}_X(\tau_1) - \overline{F}_X(\tau_2)} \, dx, $$
	and substituting this into the right-hand side of (22), we get
	\begin{align*}
		& \textstyle \qquad \mathrm{Cov}\left(X, -\ln\left[\overline{F}_X(X) - \overline{F}_X(\tau_2)\right] \,\middle|\, \tau_1 \leq X \leq \tau_2 \right) \\
		& \textstyle \qquad \quad = -\int_{\tau_1}^{\tau_2} \frac{\overline{F}_X(x) - \overline{F}_X(\tau_2)}{\overline{F}_X(\tau_1) - \overline{F}_X(\tau_2)} \ln\left[\overline{F}_X(x) - \overline{F}_X(\tau_2)\right] \, dx \\
		& \textstyle \qquad \qquad + \ln\left[\overline{F}_X(\tau_1) - \overline{F}_X(\tau_2)\right] \int_{\tau_1}^{\tau_2} \frac{\overline{F}_X(x) - \overline{F}_X(\tau_2)}{\overline{F}_X(\tau_1) - \overline{F}_X(\tau_2)} \, dx \\
		& \textstyle \qquad \quad = H(X; \tau_1, \tau_2).
	\end{align*}
\end{proof}

From Theorem 2, we can directly obtain the following corollary.
\begin{corollary}
	Let $X$ be a non-negative absolutely continuous random variable. Then it holds
	$$ \textstyle \mathcal{E}(X;t) = \mathrm{Cov}(X, -\ln \overline{F}_X(X) \mid X > t). $$
\end{corollary} 

For an absolutely non-negative continuous random variable $X$ with survival function $\overline{F}_X(x)=1-F_X(x)$ and finite mean value, let us now recall that the \textit{relevation transform} \# (where the symbol \# denotes the relevation transform of $F_X$ by itself) allows us to define a new random variable, denoted as $X\# X$, having survival function 
\begin{equation} \textstyle 
	(\overline{F}_X\# \overline{F}_X)(x)=\overline{F}_{X\# X}(x)=\overline{F}_X(x)+\overline{F}_X(x)\int_0^x \frac{f_X(u)}{\overline{F}_X(u)}\,du,\quad x\geq 0.
\end{equation}
The relevation transform was introduced by Krakowski \cite{Krakowski1973}. For further results, see \cite{Baxter1982}, \cite{Kapodistria2012}, \cite{Psarrakos2013}, \cite{Burkschat2014}, \cite{DiCrescenzo2015}, and \cite{Psarrakos2018} as well as the references therein.

Based on (23), we can define the revelation transform of the doubly truncated random variable $X_{\tau_1,\tau_2}=X|\tau_1 \leq X \leq \tau_2$, i.e., the relevation transform of $F_{X_{\tau_1,\tau_2}}$ by itself. Thus, we have the following definition.
\begin{definition}
	Let $X$ be a non-negative absolutely continuous random variable. The survival function of the relevation transform of $X_{\tau_1,\tau_2}\# X_{\tau_1,\tau_2}$, where $X_{\tau_1,\tau_2}=X|\tau_1 \leq X \leq \tau_2$, is given by
	\begin{align}
		(\overline{F}_{X_{\tau_1,\tau_2}}\# \overline{F}_{X_{\tau_1,\tau_2}})(x) & \textstyle =\overline{F}_{X_{\tau_1,\tau_2}\# X_{\tau_1,\tau_2}}(x) \nonumber \\
		& \textstyle +\overline{F}_{X_{\tau_1,\tau_2}\# X_{\tau_1,\tau_2}}(x)\int_{\tau_1}^x \frac{f_{X_{\tau_1,\tau_2}}(u)}{\overline{F}_{X_{\tau_1,\tau_2}}(u)}\,du,\quad \tau_1\leq x\leq \tau_2.
	\end{align}
\end{definition}

It can be verified that $\overline{F}_{X_{\tau_1,\tau_2}}\# \overline{F}_{X_{\tau_1,\tau_2}}$ is a survival function. Indeed, it holds
$$ \textstyle (\overline{F}_{X_{\tau_1,\tau_2}}\# \overline{F}_{X_{\tau_1,\tau_2}})(\tau_1)=1;\quad (\overline{F}_{X_{\tau_1,\tau_2}}\# \overline{F}_{X_{\tau_1,\tau_2}})(\tau_2)=0 $$
and
$$ \textstyle \frac{d}{dx}(\overline{F}_{X_{\tau_1,\tau_2}}\# \overline{F}_{X_{\tau_1,\tau_2}})(x) = -f_{X_{\tau_1,\tau_2}}(x)\int_{\tau_1}^x h_{X,1}(u,\tau_2)\,du < 0, $$
implying that $\overline{F}_{X_{\tau_1,\tau_2}}\# \overline{F}_{X_{\tau_1,\tau_2}}$ is a survival function.

From (24), we can directly obtain the following alternative representation for the survival function of $X_{\tau_1,\tau_2}\# X_{\tau_1,\tau_2}$.
\begin{equation} \textstyle 
	(\overline{F}_{X_{\tau_1,\tau_2}}\# \overline{F}_{X_{\tau_1,\tau_2}})(x)=\overline{F}_{X_{\tau_1,\tau_2}}(x)+\overline{F}_{X_{\tau_1,\tau_2}}(x)\int_{\tau_1}^x h_{X,1}(u,\tau_2)\,du.
\end{equation}

Also, by setting $\Lambda_X(x;\tau_1,\tau_2)$ to the cumulative hazard rate at the age $x$ for the doubly truncated random variable $X_{\tau_1,\tau_2}$, i.e.,
$$ \textstyle \Lambda_X(x;\tau_1,\tau_2)=\int_{\tau_1}^x h_{X,1}(u,\tau_2)\,du,\quad \tau_1\leq x\leq \tau_2, $$
or equivalently by Lemma 1 (i)
$$ \textstyle \Lambda_X(x;\tau_1,\tau_2)=-\ln \overline{F}_{X_{\tau_1,\tau_2}}(x) = -\ln\left( \frac{\overline{F}_X(x)-\overline{F}_X(\tau_2)}{\overline{F}_X(\tau_1)-\overline{F}_X(\tau_2)} \right),\quad \tau_1\leq x\leq \tau_2, $$
then from (25) it also follows that
$$ \textstyle (\overline{F}_{X_{\tau_1,\tau_2}}\# \overline{F}_{X_{\tau_1,\tau_2}})(x) = \left[1+\Lambda_X(x;\tau_1,\tau_2)\right] \overline{F}_{X_{\tau_1,\tau_2}}(x). $$
For $\tau_1=0$, $\tau_2\to\infty$, the last relation is reduced to the well-known identity (see, for example, rel. (10) in \cite{Psarrakos2018})
$$ \textstyle (\overline{F}_X\# \overline{F}_X)(x) = [1+\Lambda_X(x)]\overline{F}_X(x). $$
From (23), it follows that
\begin{equation} \textstyle 
	\textstyle \mathcal{E}(X)=\int_0^\infty (\overline{F}_X\# \overline{F}_X)(x)\,dx - E(X),
\end{equation}
that is, the CRE of $X$ can be written as the difference between the mean values of $X\# X$ and $X$. The next theorem provides an alternative representation for $H(X;\tau_1,\tau_2)$ based on the relevation transform of $F_{X_{\tau_1,\tau_2}}$ by itself, which also generalizes (26).

\begin{theorem}
	If $X$ is a non-negative absolutely continuous random variable, then, for all $(\tau_1,\tau_2)\in D_X$
	\begin{equation} \textstyle 
		H(X;\tau_1,\tau_2)=\int_{\tau_1}^{\tau_2} (\overline{F}_{X_{\tau_1,\tau_2}}\# \overline{F}_{X_{\tau_1,\tau_2}})(x)\,dx - m_{X,1}(\tau_1,\tau_2).
	\end{equation}
\end{theorem}

\begin{proof}
	Using (25) and (8) we obtain
	$$ \textstyle \int_{\tau_1}^{\tau_2} (\overline{F}_{X_{\tau_1,\tau_2}}\# \overline{F}_{X_{\tau_1,\tau_2}})(x)\,dx = m_{X,1}(\tau_1,\tau_2)+\int_{\tau_1}^{\tau_2} \overline{F}_{X_{\tau_1,\tau_2}}(x) \int_{\tau_1}^x h_{X,1}(u,\tau_2)\,du\,dx. $$
	Applying Fubini’s Theorem, the last relationship can be rewritten as
	\begin{align*}
		& \textstyle \int_{\tau_1}^{\tau_2} (\overline{F}_{X_{\tau_1,\tau_2}}\# \overline{F}_{X_{\tau_1,\tau_2}})(x)\,dx = m_{X,1}(\tau_1,\tau_2) + \int_{\tau_1}^{\tau_2} \int_u^{\tau_2} h_{X,1}(u,\tau_2) \overline{F}_{X_{\tau_1,\tau_2}}(x)\,dx\,du \\
		& \textstyle \qquad \qquad = m_{X,1}(\tau_1,\tau_2) + \int_{\tau_1}^{\tau_2} \frac{f_X(u)}{\overline{F}_X(\tau_1)-\overline{F}_X(\tau_2)} \left(\int_u^{\tau_2} \frac{\overline{F}_X(x)-\overline{F}_X(\tau_2)}{\overline{F}_X(u)-\overline{F}_X(\tau_2)}\,dx\right) du \\
		& \textstyle \qquad \qquad = m_{X,1}(\tau_1,\tau_2) + \frac{1}{\overline{F}_X(\tau_1)-\overline{F}_X(\tau_2)} \int_{\tau_1}^{\tau_2} m_{X,1}(u,\tau_2) f_X(u)\,du \\
		& \textstyle \qquad \qquad = m_{X,1}(\tau_1,\tau_2) + \mathbb{E}[m_{X,1}(X,\tau_2) \mid \tau_1 \leq X\leq \tau_2],
	\end{align*}
	and hence the result follows from Theorem 1.
\end{proof}
It is easy to see that for $\tau_1=0$, $\tau_2\to\infty$, (27) is reduced to (26).

\begin{remark} 
	(i) If we consider the random variable $X_{t}=X|X>t$, then from Theorem 3 it follows that
	$$ \textstyle \mathcal{E}(X;t)=\int_t^\infty (\overline{F}_{X_{t}}\# \overline{F}_{X_{t}})(x)\,dx - m_X(t). $$
	
	\noindent (ii) For the residual lifetime random variable $X_{(t)} = X - t | X > t$, it can be easily shown that
	$$ \textstyle \int_0^\infty (\overline{F}_{X_{(t)}}\# \overline{F}_{X_{(t)}})(x)\,dx = m_X(t) + \mathbb{E}[m_X(X) \mid X > t], $$
	and since the dynamic cumulative residual entropy $\mathcal{E}(X;t)$ satisfies (see \cite{Asadi2007})
	$$\mathcal{E}(X;t)=\mathbb{E}[m_X(X) \mid X > t],$$
	it also follows that
	$$ \textstyle \mathcal{E}(X;t)=\int_0^\infty (\overline{F}_{X_{(t)}}\# \overline{F}_{X_{(t)}})(x)\,dx - m_X(t). $$
\end{remark}

The following results consider the evaluation of CRIE under increasing transformations.
 
\begin{theorem}
	Let $X$ be a non-negative absolutely continuous random variable and $Y=\varphi(X)$, where $\varphi$ is an increasing and differentiable function with continuous derivative $\varphi'$. Then, for all $(\tau_1,\tau_2)\in D_X$
	\begin{equation*}
		\textstyle H(Y;\tau_1,\tau_2) = -\int_{\varphi^{-1}(\tau_1)}^{\varphi^{-1}(\tau_2)} \varphi'(x)\frac{\overline{F}_X(x)-\overline{F}_X(\varphi^{-1}(\tau_2))}{\overline{F}_X(x)-\overline{F}_X(\varphi^{-1}(\tau_2))} \ln\left(\frac{\overline{F}_X(x)-\overline{F}_X(\varphi^{-1}(\tau_2))}{\overline{F}_X(\varphi^{-1}(\tau_1))-\overline{F}_X(\varphi^{-1}(\tau_2))}\right)\,dx.
	\end{equation*}
\end{theorem}

\begin{proof}
	From (5) we have
	\begin{equation*}
		\textstyle H(Y;\tau_1,\tau_2) = -\int_{\tau_1}^{\tau_2} \frac{\overline{F}_Y(y)-\overline{F}_Y(\tau_2)}{\overline{F}_Y(\tau_1)-\overline{F}_Y(\tau_2)} \ln\left(\frac{\overline{F}_Y(y)-\overline{F}_Y(\tau_2)}{\overline{F}_Y(\tau_1)-\overline{F}_X(\tau_2)}\right)\,dy.
	\end{equation*}
	Then, by setting $y=\varphi(x)$, we obtain $\overline{F}_Y(y)=\overline{F}_X(\varphi^{-1}(y))$, implying that
	{\small 
		\begin{equation*}
			\textstyle H(Y;\tau_1,\tau_2) = -\int_{\varphi^{-1}(\tau_1)}^{\varphi^{-1}(\tau_2)} \frac{\overline{F}_X(x)-\overline{F}_X(\varphi^{-1}(\tau_2))}{\overline{F}_X(\varphi^{-1}(\tau_1))-\overline{F}_X(\varphi^{-1}(\tau_2))} \ln\left(\frac{\overline{F}_X(x)-\overline{F}_X(\varphi^{-1}(\tau_2))}{\overline{F}_X(\varphi^{-1}(\tau_1))-\overline{F}_X(\varphi^{-1}(\tau_2))}\right) \varphi'(x)\,dx,
	\end{equation*} } and hence, the result follows.
\end{proof}

\begin{proposition}
	Let $X$ be a non-negative absolutely continuous random variable and $Y=aX+b$, where $a>0$, $b\geq 0$. Then
	$$ \textstyle H(Y;\tau_1,\tau_2) = a H(X;\frac{\tau_1-b}{a},\frac{\tau_2-b}{a}). $$
\end{proposition}

\begin{proof}
	Let $y=\varphi(x)=ax+b$. Then, $\varphi^{-1}(t_i)=\frac{t_i - b}{a}$ and since $a>0$, the function $\varphi$ is increasing with $\varphi'(x)=a$. Applying Theorem 4, the result follows.
\end{proof}

In the next section, we shall give bounds for $H(Y;\tau_1,\tau_2)$, some of which are derived by using the results of this section.

\section{Bounds for $\boldsymbol{H(X;\tau_1,\tau_2)}$} \label{sec4}

Since for many distributions it is very difficult to analytically compute the measure $H(X;\tau_1,\tau_2 )$, it is useful to obtain upper and lower bounds for it. In Theorem 5 by using the well-known bound
\begin{equation*} \textstyle
	\mathrm{Cov}(X,Y)\leq\sqrt{\mathrm{Var}(X)} \sqrt{\mathrm{Var}(Y)},
\end{equation*}
for two random variables $X$ and $Y$, we obtain an upper bound for $H(X;\tau_1,\tau_2 )$ in terms of the conditional variance $\mathrm{Var}(X|\tau_1 \leq X\leq \tau_2)$.

\begin{theorem} 
	Let $X$ be a non-negative absolutely continuous random variable. Then, for all $(\tau_1,\tau_2)\in D_X$
	\begin{equation} \textstyle
		H(X;\tau_1,\tau_2 )\leq\sqrt{\mathrm{Var}(X|\tau_1 \leq X\leq \tau_2)}.
	\end{equation}
\end{theorem}

\begin{proof}
	From Theorem 2 it follows that
	\begin{align}
		H(X;\tau_1,\tau_2 ) & =  \textstyle\mathrm{Cov}(X,-\ln[\overline{F}_X (X)-\overline{F}_X (\tau_2)]|\tau_1 \leq X\leq \tau_2 ) \nonumber \\
		&  \textstyle \leq \sqrt{\mathrm{Var}(X|\tau_1 \leq X\leq \tau_2)} \sqrt{\mathrm{Var}(-\ln[\overline{F}_X (X)-\overline{F}_X (\tau_2)]|\tau_1 \leq X\leq \tau_2 )} \nonumber \\
		& =  \textstyle \sqrt{\mathrm{Var}(X|\tau_1 \leq X\leq \tau_2)} \sqrt{\mathrm{Var}(\ln[\overline{F}_X (X)-\overline{F}_X (\tau_2)]|\tau_1 \leq X\leq \tau_2 )}.
	\end{align}
	Since
	\begin{align*}
		&  \textstyle\mathbb{E}\{(\ln[\overline{F}_X (X)-\overline{F}_X (\tau_2)])^2 |\tau_1 \leq X\leq \tau_2 \} \\
		&  \textstyle\qquad = -\int_{\tau_1}^{\tau_2} \frac{(\ln[\overline{F}_X (x)-\overline{F}_X (\tau_2)])^2}{\overline{F}_X (\tau_1 )-\overline{F}_X (\tau_2 )} \left( \overline{F}_X (x)-\overline{F}_X (\tau_2 ) \right)' dx \\
		&  \textstyle\qquad = (\ln[\overline{F}_X (\tau_1 )-\overline{F}_X (\tau_2)])^2 -2 \int_{\tau_1}^{\tau_2} \frac{\ln[\overline{F}_X (x)-\overline{F}_X (\tau_2 )]}{\overline{F}_X (\tau_1 )-\overline{F}_X (\tau_2 )} f_X (x)dx \\
		&  \textstyle\qquad = (\ln[\overline{F}_X (\tau_1 )-\overline{F}_X (\tau_2)])^2 -2\mathbb{E}(\ln[\overline{F}_X (X)-\overline{F}_X (\tau_2)]|\tau_1 \leq X\leq \tau_2 ) \\
		&  \textstyle \qquad = (\ln[\overline{F}_X (\tau_1 )-\overline{F}_X (\tau_2)])^2 -2\ln[\overline{F}_X (\tau_1 )-\overline{F}_X (\tau_2)]+2,
	\end{align*}
	where the last equality follows from (20). Using again (20) and the last relationship, it follows that
	\begin{align}
		& \quad \mathrm{Var}(\ln[\overline{F}_X (X)-\overline{F}_X (\tau_2)]|\tau_1 \leq X\leq \tau_2 ) \nonumber \\
		& \qquad = \mathbb{E}\{(\ln[\overline{F}_X (X)-\overline{F}_X (\tau_2)])^2 |\tau_1 \leq X\leq \tau_2 \} \nonumber \\
		& \qquad \quad -\{\mathbb{E}(\ln[\overline{F}_X (X)-\overline{F}_X (\tau_2)]|\tau_1 \leq X\leq \tau_2 )\}^2 \nonumber \\
		& \qquad = (\ln[\overline{F}_X (\tau_1 )-\overline{F}_X (\tau_2)])^2 -2\ln[\overline{F}_X (\tau_1 )-\overline{F}_X (\tau_2)]+2 \nonumber \\
		& \qquad \quad -(\ln[\overline{F}_X (\tau_1 )-\overline{F}_X (\tau_2)]-1)^2 \nonumber \\
		& \qquad = 1.
	\end{align}
	Now, (28) follows directly from (29) and (30).
\end{proof}

From Theorem 5 it follows that if $\mathrm{Var}(X|\tau_1 \leq X\leq \tau_2)<\infty$, then $H(X;\tau_1,\tau_2 )$ is finite. Also, from Theorem 5 we get the following corollary.

\begin{corollary}
	Let $X$ be a non-negative absolutely continuous random variable. Then
	$$ \textstyle \mathcal{E}(X;t)\leq\sqrt{\mathrm{Var}(X|X>t)}. $$
\end{corollary}

Using again Theorem 2 we can also obtain a different upper bound for $H(X;\tau_1,\tau_2 )$ in terms of $\mathbb{E}[|X-\mathbb{E}(X|\tau_1 \leq X\leq \tau_2 )||\tau_1 \leq X\leq \tau_2 ]$.

\begin{theorem}
	Let $X$ be a non-negative absolutely continuous random variable. Then, for all $(\tau_1,\tau_2)\in D_X$
	\begin{align}
		H(X;\tau_1,\tau_2 ) & \textstyle \leq 2\mathbb{E}[|X-\mathbb{E}(X|\tau_1 \leq X\leq \tau_2 )|\ln|X-\mathbb{E}(X|\tau_1 \leq X\leq \tau_2 )||\tau_1 \leq X\leq \tau_2 ] \nonumber \\
		& \textstyle \quad + \frac{4e^{-1}}{\sqrt{\overline{F}_X (\tau_1 )-\overline{F}_X (\tau_2)}}.
	\end{align}
\end{theorem}

\begin{proof}
	From (18) and (19) it follows that $H(X;\tau_1,\tau_2 )$ can be equivalently rewritten as
	$$ \textstyle H(X;\tau_1,\tau_2 ) = \mathbb{E}\left\{[X-\mathbb{E}(X|\tau_1 \leq X\leq \tau_2 )](-\ln[\overline{F}_X (X)-\overline{F}_X (\tau_2)])|\tau_1 \leq X\leq \tau_2 \right\}, $$
	and following similar steps as in Rao \cite[Section 6]{Rao2005} we obtain
	\begin{align}
		H(X;\tau_1,\tau_2 ) \leq & \textstyle 2\mathbb{E}(|X-\mathbb{E}(X|\tau_1 \leq X\leq \tau_2 )|\ln|X-\mathbb{E}(X|\tau_1 \leq X\leq \tau_2 )||\tau_1 \leq X\leq \tau_2 ) \nonumber \\
		& \textstyle \ + 2\mathbb{E}\left\{-\exp\left(-\sqrt{\ln[\overline{F}_X (X)-\overline{F}_X (\tau_2)]}-1\right)|\tau_1 \leq X\leq \tau_2 \right\}.
	\end{align}
	It holds
	\begin{align*}
		& \textstyle \mathbb{E}\left\{-\exp\left(-\sqrt{\ln[\overline{F}_X (X)-\overline{F}_X (\tau_2)]}-1\right)|\tau_1 \leq X\leq \tau_2 \right\} \\
		& \textstyle \qquad = e^{-1} \int_{\tau_1}^{\tau_2} \frac{[\overline{F}_X (x)-\overline{F}_X (\tau_2 )]^{-1/2}}{\overline{F}_X (\tau_1 )-\overline{F}_X (\tau_2)} f_X (x)\,dx \\
		& \textstyle \qquad = e^{-1} \int_0^{\overline{F}_X (\tau_1 )-\overline{F}_X (\tau_2)} \frac{u^{-1/2}}{\overline{F}_X (\tau_1 )-\overline{F}_X (\tau_2)}\,du \\
		& \textstyle \qquad = \frac{2e^{-1}}{\sqrt{\overline{F}_X (\tau_1 )-\overline{F}_X (\tau_2)}},
	\end{align*}
	and hence the bound in (31) follows directly from (32) and the last relationship.
\end{proof}

When $\tau_1=0$ and $\tau_2 \to \infty$, the bound in (31) is reduced to the upper bound for $\mathcal{E}(X)$ given in Rao \cite[Section 6]{Rao2005}. When $\tau_2 \to \infty$, from Theorem 6 we immediately obtain the following corollary.

\begin{corollary}
	Let $X$ be a non-negative absolutely continuous random variable. Then
	$$ \textstyle \mathcal{E}(X;t) \leq 2\mathbb{E}(|X - \mathbb{E}(X \mid X > t)| \ln |X - \mathbb{E}(X \mid X > t)| \mid X > t) + \frac{4e^{-1}}{\sqrt{\overline{F}_X(t)}}. $$
\end{corollary}

\begin{proposition}
	Let $X$ be a non-negative absolutely continuous random variable. Then, for all $(\tau_1,\tau_2)\in D_X$
	$$ \textstyle H(X;\tau_1,\tau_2) \leq m_{X,1}(\tau_1,\tau_2) \ln\left(\frac{\tau_2 - \tau_1}{m_{X,1}(\tau_1,\tau_2)}\right). $$
\end{proposition}

\begin{proof}
	From Log-Sum inequality we have
	\begin{align*}
		& \textstyle \int_{\tau_1}^{\tau_2} [\overline{F}_X(x) - \overline{F}_X(\tau_2)] \ln\left( \frac{\overline{F}_X(x) - \overline{F}_X(\tau_2)}{\overline{F}_X(\tau_1) - \overline{F}_X(\tau_2)} \right) dx \\
		& \textstyle \qquad \geq \left( \int_{\tau_1}^{\tau_2} [\overline{F}_X(x) - \overline{F}_X(\tau_2)] dx \right) \ln \left( \frac{ \int_{\tau_1}^{\tau_2} [\overline{F}_X(x) - \overline{F}_X(\tau_2)] dx }{ (\tau_2 - \tau_1)[\overline{F}_X(\tau_1) - \overline{F}_X(\tau_2)] } \right).
	\end{align*}
	Using (8), the above relationship yields
	\begin{align*}
		& \textstyle \int_{\tau_1}^{\tau_2} [\overline{F}_X(x) - \overline{F}_X(\tau_2)] \ln\left( \frac{\overline{F}_X(x) - \overline{F}_X(\tau_2)}{\overline{F}_X(\tau_1) - \overline{F}_X(\tau_2)} \right) dx \\ 
		& \textstyle \qquad \geq m_{X,1}(\tau_1,\tau_2)[\overline{F}_X(\tau_1) - \overline{F}_X(\tau_2)] \ln\left( \frac{m_{X,1}(\tau_1,\tau_2)}{\tau_2 - \tau_1} \right),
	\end{align*}
	which proves the upper bound.
\end{proof}

Using the inequality $1 - \frac{1}{x} \leq \ln x \leq x - 1$ for $x \geq 0$, from the definition of $H(X;\tau_1,\tau_2)$ we can easily obtain a two-sided bound for $H(X;\tau_1,\tau_2)$ given in the next proposition.

\begin{proposition}
	Let $X$ be a non-negative absolutely continuous random variable. Then, for all $(\tau_1,\tau_2)\in D_X$
	\begin{align*}
		\textstyle \frac{1}{[\overline{F}_X(\tau_1) - \overline{F}_X(\tau_2)]^2} \int_{\tau_1}^{\tau_2} [\overline{F}_X(x) - \overline{F}_X(\tau_2)][\overline{F}_X(\tau_1) - \overline{F}_X(x)] dx \leq H(X;\tau_1,\tau_2) \leq m_{X,2}(\tau_1,\tau_2).
	\end{align*}
\end{proposition}

\begin{remark}  
	(i) Note that the lower bound in Proposition 3 can be equivalently rewritten as
	$$ \textstyle H(X;\tau_1,\tau_2) \geq \int_{\tau_1}^{\tau_2} \overline{F}_{X_{\tau_1,\tau_2}}(x) F_{X_{\tau_1,\tau_2}}(x) dx. $$
	
	\noindent (ii) When $\tau_1=0$, $\tau_2 \to \infty$, from Proposition 3 we obtain that
	$$ \textstyle \mathcal{E}(X) \geq \int_0^\infty \overline{F}_X(x) F_X(x) dx, $$
	which is also given in Rao \cite[rel. (5)]{Rao2005}.
	
	\noindent (iii) When $\tau_2 \to \infty$, from Proposition 2 we obtain the following lower bound for the dynamic cumulative residual entropy of $X$.
	$$ \textstyle \mathcal{E}(X;t) \geq m_X(t) - \frac{1}{[\overline{F}_X(t)]^2} \int_t^\infty [\overline{F}_X(x)]^2 dx. $$
\end{remark}

A different lower bound for $H(X;\tau_1,\tau_2)$ than that obtained in Proposition 3 is given in Proposition 4.

\begin{proposition}
	Let $X$ be a non-negative absolutely continuous random variable. Then, for all $(\tau_1,\tau_2)\in D_X$
	$$ \textstyle H(X;\tau_1,\tau_2) \geq \frac{1}{\overline{F}_X(\tau_1) - \overline{F}_X(\tau_2)} \mathbb{E}\left\{ [\overline{F}_X(X) - \overline{F}_X(\tau_2)] m_{X,1}(X,\tau_2) \mid \tau_1 \leq X \leq \tau_2 \right\}. $$
\end{proposition}

\begin{proof}
	For $u \geq \tau_1$ it holds $\overline{F}_X(u) \leq \overline{F}_X(\tau_1)$, and thus
	$$ \textstyle \int_{\tau_1}^x h_{X,1}(u,\tau_2) du = \int_{\tau_1}^x \frac{f_X(u)}{\overline{F}_X(u) - \overline{F}_X(\tau_2)} du \geq \int_{\tau_1}^x \frac{f_X(u)}{\overline{F}_X(\tau_1) - \overline{F}_X(\tau_2)} du. $$
	Therefore, from (25) we get
	\begin{align*}
		\textstyle \int_{\tau_1}^{\tau_2} \left( \overline{F}_{X_{\tau_1,\tau_2}} \# \overline{F}_{X_{\tau_1,\tau_2}} \right)(x) dx \geq \int_{\tau_1}^{\tau_2} \overline{F}_{X_{\tau_1,\tau_2}}(x) dx + \frac{1}{\overline{F}_X(\tau_1) - \overline{F}_X(\tau_2)} \int_{\tau_1}^{\tau_2} \overline{F}_{X_{\tau_1,\tau_2}}(x) \int_{\tau_1}^x f_X(u) du \, dx,
	\end{align*}
	or equivalently, by using Theorem 3
	$$ \textstyle H(X;\tau_1,\tau_2) \geq \frac{1}{\overline{F}_X(\tau_1) - \overline{F}_X(\tau_2)} \int_{\tau_1}^{\tau_2} \overline{F}_{X_{\tau_1,\tau_2}}(x) \int_{\tau_1}^x f_X(u) du \, dx. $$
	Applying Fubini’s Theorem, it follows that
	$$ \textstyle H(X;\tau_1,\tau_2) \geq \frac{1}{\overline{F}_X(\tau_1) - \overline{F}_X(\tau_2)} \int_{\tau_1}^{\tau_2} f_X(u) \int_u^{\tau_2} \overline{F}_{X_{\tau_1,\tau_2}}(x) dx \, du, $$
	and since
	\begin{align*}
		\textstyle \int_u^{\tau_2} \overline{F}_{X_{\tau_1,\tau_2}}(x) dx = \textstyle \frac{1}{\overline{F}_X(\tau_1) - \overline{F}_X(\tau_2)} \int_u^{\tau_2} [\overline{F}_X(x) - \overline{F}_X(\tau_2)] dx = \textstyle \frac{\overline{F}_X(u) - \overline{F}_X(\tau_2)}{\overline{F}_X(\tau_1) - \overline{F}_X(\tau_2)} m_{X,1}(u,\tau_2),
	\end{align*}
	we obtain
	$$ \textstyle H(X;\tau_1,\tau_2) \geq \frac{1}{[\overline{F}_X(\tau_1) - \overline{F}_X(\tau_2)]^2} \int_{\tau_1}^{\tau_2} [\overline{F}_X(u) - \overline{F}_X(\tau_2)] m_{X,1}(u,\tau_2) f_X(u) du, $$
	which completes the proof.
\end{proof}

\begin{corollary}
	Let $X$ be a non-negative absolutely continuous random variable. If $h_{X,1}(x,y)$ is decreasing in $x$ for any fixed value of $y$, then, for all $(\tau_1,\tau_2)\in D_X$
	$$ \textstyle H(X;\tau_1,\tau_2) \geq \int_{\tau_1}^{\tau_2} \left( \frac{\overline{F}_X(x) - \overline{F}_X(\tau_2)}{\overline{F}_X(\tau_1) - \overline{F}_X(\tau_2)} \right)^2 dx. $$
\end{corollary}

\begin{proof}
	If $h_{X,1}(x,y)$ is decreasing in $x$, then $m_{X,1}(x,y)$ is increasing in $x$ (see Lemma 2 (ii)), and hence from (9) it follows that
	$$ m_{X,1}(x,y) \geq \frac{1}{h_{X,1}(x,y)}. $$
	Then, from Proposition 4 we have
	$$ \textstyle H(X;\tau_1,\tau_2) \geq \frac{1}{[\overline{F}_X(\tau_1) - \overline{F}_X(\tau_2)]^2} \int_{\tau_1}^{\tau_2} [\overline{F}_X(u) - \overline{F}_X(\tau_2)] \frac{1}{h_{X,1}(u,\tau_2)} f_X(u) du, $$
	and using (6) the result follows.
\end{proof}

Since $h_{X,1}(x,\infty) = h_X(x)$, from Corollary 4 we immediately obtain the following corollary.

\begin{corollary}
	Let $X$ be a non-negative absolutely continuous random variable. If $X$ is DFR, then
	$$ \textstyle \mathcal{E}(X;t) \geq \int_t^\infty \left( \frac{\overline{F}_X(x)}{\overline{F}_X(t)} \right)^2 dx, $$
	and
	$$ \textstyle \mathcal{E}(X) \geq \int_0^\infty \left( \overline{F}_X(x) \right)^2 dx. $$
\end{corollary}

\begin{remark}  
	(i) Let $L_C = \int_0^\infty (\overline{F}_X(x))^2 dx$ be the lower bound for $\mathcal{E}(X)$ given in Corollary 5, and $L_R = \int_0^\infty \overline{F}_X(x) F_X(x) dx$ be the general lower bound for $\mathcal{E}(X)$ given by \cite{Rao2005}. Suppose that $X$ is absolutely continuous and DFR. Let $x_m$ be the median of $X$, i.e., $x_m$ is the solution of $F_X(x_m) = \overline{F}_X(x_m) = \frac{1}{2}$. It holds that $L_C \geq (\leq) L_R$ if and only if $F_X(x) \leq (\geq) \frac{1}{2}$, i.e., $F_X(x) \leq (\geq) F_X(x_m)$, and since $F_X(x)$ is increasing in $x$, it follows that $L_C \geq (\leq) L_R$ if and only if $x \leq (\geq) x_m$. Therefore, if $X$ is non-negative absolutely continuous and DFR, then the bound $L_C$ is better than $L_R$ for $x < x_m$, and $L_R$ is better than $L_C$ for $x > x_m$.
	
	\noindent (ii) Recently, Jahanshahi et al. \cite{Jahanshahi2020} proposed an alternative measure of uncertainty for non-negative random variables called cumulative residual extropy (CRJ), which is denoted by
	$$ \textstyle \text{CRJ}(X) = -\frac{1}{2} \int_0^\infty (\overline{F}_X(x))^2 dx. $$
	Therefore, from Corollary 5 it follows that if the non-negative absolutely continuous random variable $X$ is DFR, then $\mathcal{E}(X)$ and $\text{CRJ}(X)$ satisfy the relationship
	$$ \mathcal{E}(X) \geq -2 \text{CRJ}(X). $$
\end{remark}

In Proposition 5, as in \cite{Rao2004}, we obtain a lower bound for $H(X;\tau_1,\tau_2)$ in terms of the entropy $S(X;\tau_1,\tau_2)$, as well as an upper bound using the Log-Sum inequality.

\begin{proposition}
	Let $X$ be a non-negative absolutely continuous random variable. Then, for all $(\tau_1,\tau_2)\in D_X$
	$$ \textstyle H(X;\tau_1,\tau_2 )\geq C\exp\{S(X;\tau_1,\tau_2 )\}, $$
	where $C=\exp\left\{\int_0^1 \ln(u|\ln u|)\,du\right\}\approx 0.2065$, $S(X;\tau_1,\tau_2 )$ is defined by (4), and
	$$ H(X;\tau_1,\tau_2 )\leq m_{X,1}(\tau_1,\tau_2 )\ln\left(\frac{\tau_2-\tau_1}{m_{X,1}(\tau_1,\tau_2 )}\right). $$
\end{proposition}

\begin{proof} 
	Using again the Log-Sum inequality, then as in Rao et al. \cite{Rao2004}, we have
	\begin{equation*}
		\textstyle \int_{\tau_1}^{\tau_2} \frac{f_X(x)}{\overline{F}_X(\tau_1)-\overline{F}_X(\tau_2)} \ln\left(\frac{\frac{f_X(x)}{\overline{F}_X(\tau_1)-\overline{F}_X(\tau_2)}}{\frac{\overline{F}_X(x)-\overline{F}_X(\tau_2)}{\overline{F}_X(\tau_1)-\overline{F}_X(\tau_2)}} \left| \frac{\overline{F}_X(x)-\overline{F}_X(\tau_2)}{\overline{F}_X(\tau_1)-\overline{F}_X(\tau_2)} \right| \right)dx \geq \ln\left(\frac{1}{H(X;\tau_1,\tau_2 )}\right),
	\end{equation*}
	implying that
	\begin{align*}
		\textstyle -S(X;\tau_1,\tau_2 ) - \int_{\tau_1}^{\tau_2} \frac{f_X(x)}{\overline{F}_X(\tau_1)-\overline{F}_X(\tau_2)} \ln\left( \frac{\overline{F}_X(x)-\overline{F}_X(\tau_2)}{\overline{F}_X(\tau_1)-\overline{F}_X(\tau_2)} \left| \frac{\overline{F}_X(x)-\overline{F}_X(\tau_2)}{\overline{F}_X(\tau_1)-\overline{F}_X(\tau_2)} \right| \right)dx \geq \ln\left(\frac{1}{H(X;\tau_1,\tau_2 )}\right),
	\end{align*}
	which yields
	\begin{align*}
		\textstyle \ln H(X;\tau_1,\tau_2 ) \geq \int_{\tau_1}^{\tau_2} \frac{f_X(x)}{\overline{F}_X(\tau_1)-\overline{F}_X(\tau_2)} \ln\left( \frac{\overline{F}_X(x)-\overline{F}_X(\tau_2)}{\overline{F}_X(\tau_1)-\overline{F}_X(\tau_2)} \left| \frac{\overline{F}_X(x)-\overline{F}_X(\tau_2)}{\overline{F}_X(\tau_1)-\overline{F}_X(\tau_2)} \right| \right)dx + S(X;\tau_1,\tau_2 ).
	\end{align*}
	By setting
	$$ \textstyle u = \frac{\overline{F}_X(x)-\overline{F}_X(\tau_2)}{\overline{F}_X(\tau_1)-\overline{F}_X(\tau_2)}, $$
	then the integral becomes
	$$ \textstyle \int_0^1 \ln(u|\ln u|)\,du, $$
	and hence the lower bound follows. Using again the Log-Sum inequality, we obtain
	\begin{align*}
		& \textstyle \int_{\tau_1}^{\tau_2} [\overline{F}_X(x)-\overline{F}_X(\tau_2)] \ln\left( \frac{\overline{F}_X(x)-\overline{F}_X(\tau_2)}{\overline{F}_X(\tau_1)-\overline{F}_X(\tau_2)} \right)dx \\
		& \textstyle \qquad \geq \left(\int_{\tau_1}^{\tau_2} [\overline{F}_X(x)-\overline{F}_X(\tau_2)]dx\right) \ln\left( \frac{\int_{\tau_1}^{\tau_2} [\overline{F}_X(x)-\overline{F}_X(\tau_2)]dx}{(\tau_2-\tau_1)[\overline{F}_X(\tau_1)-\overline{F}_X(\tau_2)]} \right) \\
		& \textstyle \qquad = m_{X,1}(\tau_1,\tau_2)[\overline{F}_X(\tau_1)-\overline{F}_X(\tau_2)] \ln\left( \frac{m_{X,1}(\tau_1,\tau_2)}{\tau_2 - \tau_1} \right),
	\end{align*}
	which yields immediately the upper bound.
\end{proof}

Hence, one can use the lower bounds for $S(X;\tau_1,\tau_2)$ (see, e.g., Moharana and Kayal\cite{Moharana2020}) to obtain the corresponding bounds for $H(X;\tau_1,\tau_2)$.
In Proposition 6, we obtain a relationship between the CRIE of two non-negative and absolutely continuous random variables ordered in the usual likelihood ratio order.
 
\begin{proposition}
	If $X$ and $Y$ are non-negative absolutely continuous random variables with finite doubly truncated mean residual lifetimes $m_{X,1}(\tau_1,\tau_2) $ and $ m_{Y,1}(\tau_1,\tau_2)$, respectively, such that $X \leq_{\text{lr}} Y$, then, for all $(\tau_1,\tau_2) \in D_X \cap D_Y$
	$$ H(X;\tau_1,\tau_2 ) \leq H(Y;\tau_1,\tau_2 ) - m_{X,1}(\tau_1,\tau_2 ) \ln\left( \frac{m_{X,1}(\tau_1,\tau_2 )}{m_{Y,1}(\tau_1,\tau_2 )} \right). $$
\end{proposition}

\begin{proof}
	Let $\overline{F}_Y(x)=\Pr(Y>x)$ be the survival function of $Y$. From the Log-Sum inequality we obtain
	\begin{equation}
			\textstyle \int_{\tau_1}^{\tau_2} \frac{\overline{F}_X(x)-\overline{F}_X(\tau_2)}{\overline{F}_X(\tau_1)-\overline{F}_X(\tau_2)} \ln\left( \frac{ \frac{\overline{F}_X(x)-\overline{F}_X(\tau_2)}{\overline{F}_X(\tau_1)-\overline{F}_X(\tau_2)} }{ \frac{\overline{F}_Y(x)-\overline{F}_Y(\tau_2)}{\overline{F}_Y(\tau_1)-\overline{F}_Y(\tau_2)} } \right)dx \geq m_{X,1}(\tau_1,\tau_2 ) \ln\left( \frac{m_{X,1}(\tau_1,\tau_2 )}{m_{Y,1}(\tau_1,\tau_2 )} \right),
	\end{equation}
	and thus, it follows that
	\begin{equation}
			\textstyle H(X;\tau_1,\tau_2 ) \leq -\int_{\tau_1}^{\tau_2} \frac{\overline{F}_X(x)-\overline{F}_X(\tau_2)}{\overline{F}_X(\tau_1)-\overline{F}_X(\tau_2)} \ln\left( \frac{\overline{F}_Y(x)-\overline{F}_Y(\tau_2)}{\overline{F}_Y(\tau_1)-\overline{F}_Y(\tau_2)} \right)dx - m_{X,1}(\tau_1,\tau_2 ) \ln\left( \frac{m_{X,1}(\tau_1,\tau_2 )}{m_{Y,1}(\tau_1,\tau_2 )} \right).
	\end{equation}
	Since $X\leq_{\text{lr}} Y$, it follows from Theorem 1.C.5 of Shaked and Shanthikumar \cite{Shaked2007} that
	$$ \textstyle [X \mid \tau_1 \leq X \leq \tau_2] \leq_{\text{st}} [Y \mid \tau_1 \leq Y \leq \tau_2], $$
	i.e., it holds
	\begin{equation}
		\textstyle \frac{\overline{F}_X(x)-\overline{F}_X(\tau_2)}{\overline{F}_X(\tau_1)-\overline{F}_X(\tau_2)} \leq \frac{\overline{F}_Y(x)-\overline{F}_Y(\tau_2)}{\overline{F}_Y(\tau_1)-\overline{F}_Y(\tau_2)}, \quad \text{for any } x\in[\tau_1,\tau_2].
	\end{equation}
	Since
	$$ -\ln\left( \frac{\overline{F}_Y(x)-\overline{F}_Y(\tau_2)}{\overline{F}_Y(\tau_1)-\overline{F}_Y(\tau_2)} \right) \geq 0, $$
	from (34) and (35) we obtain
	$$ \textstyle H(X;\tau_1,\tau_2 ) \leq -\int_{\tau_1}^{\tau_2} \frac{\overline{F}_Y(x)-\overline{F}_Y(\tau_2)}{\overline{F}_Y(\tau_1)-\overline{F}_Y(\tau_2)} \ln\left( \frac{\overline{F}_Y(x)-\overline{F}_Y(\tau_2)}{\overline{F}_Y(\tau_1)-\overline{F}_Y(\tau_2)} \right)dx - m_{X,1}(\tau_1,\tau_2 ) \ln\left( \frac{m_{X,1}(\tau_1,\tau_2 )}{m_{Y,1}(\tau_1,\tau_2 )} \right), $$
	and hence the result follows.
\end{proof}

\begin{remark}
	(i) Using the inequality $x\ln \frac{x}{y} \geq x - y$, $x>0$, $y>0$, we get the following weaker but simpler upper bound than that given in Proposition 6 under the assumption that $X \leq_{\text{lr}} Y$
	$$ \textstyle H(X;\tau_1,\tau_2 ) \leq H(Y;\tau_1,\tau_2 ) - [m_{X,1}(\tau_1,\tau_2 ) - m_{Y,1}(\tau_1,\tau_2 )]. $$
	
	\noindent (ii) As an information distance between two distribution functions $F_X$ and $F_Y$, Kullback and Leibler \cite{Kullback1951} proposed the following discrimination measure, also known as relative entropy of $X$ and $Y$
	$$ \textstyle I_{X,Y} = \int_0^\infty f_X(x)\ln \left(\frac{f_X(x)}{f_Y(x)}\right) dx, $$
	where $f_X(x)$ and $f_Y(x)$ are the densities of $X$ and $Y$, respectively. In addition, Ebrahimi and Kirmani \cite{Ebrahimi1996b}, and Di Crescenzo and Longobardi \cite{DiCrescenzo2004} defined measures of discrimination between two truncated distributions. The results in Proposition 6 are useful for constructing a goodness-of-fit test (see also \cite{Baratpour2012}). For this, let us define a new measure of distance between two distributions that is like Kullback–Leibler divergence (KL), but using the survival function for doubly truncated random variables rather than the density function and call it cumulative residual interval Kullback–Leibler (CRIKL) divergence, that is
	\begin{align*}
		\textstyle \text{CRIKL}(F_X,F_Y;\tau_1,\tau_2 ) = \textstyle \int_{\tau_1}^{\tau_2} \frac{\overline{F}_X(x)-\overline{F}_X(\tau_2)}{\overline{F}_X(\tau_1)-\overline{F}_X(\tau_2)} \ln\left( \frac{ \frac{\overline{F}_X(x)-\overline{F}_X(\tau_2)}{\overline{F}_X(\tau_1)-\overline{F}_X(\tau_2)} }{ \frac{\overline{F}_Y(x)-\overline{F}_Y(\tau_2)}{\overline{F}_Y(\tau_1)-\overline{F}_Y(\tau_2)} } \right) dx - [m_{X,1}(\tau_1,\tau_2 ) - m_{Y,1}(\tau_1,\tau_2 )].
	\end{align*} 
	From (33) we have
	\begin{align*}
		\textstyle \int_{\tau_1}^{\tau_2} \frac{\overline{F}_X(x)-\overline{F}_X(\tau_2)}{\overline{F}_X(\tau_1)-\overline{F}_X(\tau_2)} \ln\left( \frac{ \frac{\overline{F}_X(x)-\overline{F}_X(\tau_2)}{\overline{F}_X(\tau_1)-\overline{F}_X(\tau_2)} }{ \frac{\overline{F}_Y(x)-\overline{F}_Y(\tau_2)}{\overline{F}_Y(\tau_1)-\overline{F}_Y(\tau_2)} } \right) & \geq m_{X,1}(\tau_1,\tau_2 )\ln\left( \frac{m_{X,1}(\tau_1,\tau_2 )}{m_{Y,1}(\tau_1,\tau_2 )} \right) \\
		& \textstyle \geq m_{X,1}(\tau_1,\tau_2 ) - m_{Y,1}(\tau_1,\tau_2 ),
	\end{align*} 
	and thus, it holds $\text{CRIKL}(X,Y;\tau_1,\tau_2 ) \geq 0$ and note that in the Log-Sum inequality used in Proposition 5, equality holds if and only if
	$$ \textstyle \frac{\overline{F}_X(x)-\overline{F}_X(\tau_2)}{\overline{F}_X(\tau_1)-\overline{F}_X(\tau_2)} = \frac{\overline{F}_Y(x)-\overline{F}_Y(\tau_2)}{\overline{F}_Y(\tau_1)-\overline{F}_Y(\tau_2)}. $$
	From the last relationship, we get $m_{X,1}(\tau_1,\tau_2 ) = m_{Y,1}(\tau_1,\tau_2 )$, implying that in the last inequality, the equality holds if and only if $\overline{F}_X(x) = \overline{F}_Y(x)$. Therefore, if our aim is to test the hypothesis $$ H_0 : F_X(x)=G(x), \quad \text{vs.} \quad H_1 : F_X(x)\neq G(x), $$
	{\sloppy where $G(x)$ is an estimated distribution function, then under the null hypothesis it holds $\text{CRIKL}(F_X, G; \tau_1, \tau_2) = 0$ and a large value of $\text{CRIKL}(F_X,G;\tau_1,\tau_2 )$ leads us to reject the null hypothesis $H_0$ in favor of the alternative hypothesis $H_1$.\par}
\end{remark}
 
\begin{corollary}
	If $X$ and $Y$ are two non-negative absolutely continuous random variables with finite mean residual lifetimes $m_X(t)$ and $m_Y(t)$, respectively, such that $X \leq_{\text{hr}} Y$, then
	$$ \mathcal{E}(X;t) \leq \mathcal{E}(Y;t) - m_X(t)\ln\left(\frac{m_X(t)}{m_Y(t)}\right). $$
\end{corollary}

\begin{proof}
	If $X \leq_{\text{hr}} Y$, then from the relationship (1.B.7) of Shaked and Shanthikumar\cite{Shaked2007} it holds that $[X \mid X>t] \leq_{\text{st}} [Y \mid Y>t]$, which implies that
	$$ \textstyle \frac{\overline{F}_X(x)}{\overline{F}_X(t)} \leq \frac{\overline{F}_Y(x)}{\overline{F}_Y(t)}, \quad \text{for any } x > t, $$
	and hence using (34) with $\tau_2 \to \infty$ we get that
	\begin{align}
		\mathcal{E}(X;t) & \textstyle \leq -\int_t^\infty \frac{\overline{F}_X(x)}{\overline{F}_X(t)} \ln\left( \frac{\overline{F}_Y(x)}{\overline{F}_Y(t)} \right)dx - m_X(t)\ln\left( \frac{m_X(t)}{m_Y(t)} \right) \\
		& \textstyle \leq -\int_t^\infty \frac{\overline{F}_Y(x)}{\overline{F}_Y(t)} \ln\left( \frac{\overline{F}_Y(x)}{\overline{F}_Y(t)} \right)dx - m_X(t)\ln\left( \frac{m_X(t)}{m_Y(t)} \right) \nonumber \\
		& \textstyle = \mathcal{E}(Y;t) - m_X(t)\ln\left( \frac{m_X(t)}{m_Y(t)} \right). \nonumber
	\end{align} 
\end{proof}

When $t=0$, Corollary 6 is reduced to Proposition 2.1 of Navarro et al. \cite{Navarro2010}.

\begin{remark}
	Considering that $Y \sim \text{Exp}(\lambda), \lambda > 0$, with $\overline{F}_Y (x) = e^{-\lambda x}, x \geq 0$, and $m_Y (t) = 1/\lambda$, from (36) we obtain that
	
	\begin{align}
		\textstyle -\mathcal{E}(X;t) &\geq -\lambda \int_t^\infty (x-t) \frac{\overline{F}_X(x)}{\overline{F}_X(t)} dx + m_X(t) \ln(\lambda m_X(t)) \nonumber \\
		& \textstyle = -\lambda \frac{1}{\overline{F}_X(t)} \int_t^\infty x \overline{F}_X(x) dx + \lambda t m_X(t) + m_X(t) \ln(m_X(t)) + m_X(t) \ln \lambda \nonumber \\
		&\textstyle = -\lambda \frac{1}{2} \left[ \mathbb{E}(X^2 \mid X > t) - t^2 \right] + \lambda t m_X(t) + m_X(t) \ln(m_X(t)) + m_X(t) \ln \lambda \nonumber \\
		&\textstyle = -\lambda \frac{1}{2} \mathbb{E}[(X - t)^2 \mid X > t] + m_X(t) \ln(m_X(t)) + m_X(t) \ln \lambda.
	\end{align}
	The value of $\lambda$ that maximizes the right-hand side of (37) is
	$$ \lambda_{\max} = \frac{2 m_X(t)}{\mathbb{E}[(X - t)^2 \mid X > t]}. $$
	Substituting this into the right-hand side of (37) we get
	$$ \textstyle -\mathcal{E}(X;t) \geq -m_X(t) + m_X(t) \ln \left( \frac{2 m_X^2(t)}{\mathbb{E}[(X - t)^2 \mid X > t]} \right), $$
	and using the inequality $\ln x \geq 1 - \frac{1}{x}, x \geq 0$, the above relationship implies that
	$$ \textstyle -\mathcal{E}(X;t) \geq -m_X(t) + m_X(t) \left\{ 1 - \frac{\mathbb{E}[(X - t)^2 \mid X > t]}{2 m_X^2(t)} \right\}, $$
	and thus, we obtain that
	$$ \mathcal{E}(X;t) \leq \frac{\mathbb{E}[(X - t)^2 \mid X > t]}{2 m_X^2(t)}. $$
	Note that this bound has been proved by Asadi and Zohrevand \cite[Theorem 4.10]{Asadi2007} using a different approach.
\end{remark}

From Theorem 4, we immediately obtain the following proposition.

\begin{proposition} 
	Let $X$ be a non-negative absolutely continuous random variable and $Y = \varphi(X)$, where $\varphi$ is an increasing and differentiable function with derivative $\varphi'$. If $\varphi'(x) \geq (\leq) 1$, then, for all $(\tau_1, \tau_2) \in D_X \cap D_Y$
	$$ \textstyle H(Y;\tau_1,\tau_2) \leq (\geq) H(X;\varphi^{-1}(\tau_1), \varphi^{-1}(\tau_2)). $$
\end{proposition}

As in Rao \cite[Section 5]{Rao2005}, in the next theorem and based on CRIE, we provide an upper bound for the mean of the absolute difference between two independent and identically distributed random variables that are doubly truncated.

\begin{theorem} 
	Let $X$ and $Y$ be two non-negative absolutely continuous, independent and identically distributed random variables. Then
	$$ \textstyle \mathbb{E}[|X - \mathbb{E}(X)| \mid \tau_1 \leq X \leq \tau_2] \leq \mathbb{E}[|X - Y| \mid \tau_1 \leq X \leq \tau_2, \tau_1 \leq Y \leq \tau_2] \leq 2 H(X;\tau_1,\tau_2). $$
\end{theorem}

\begin{proof}
	Since
	$$ \textstyle \Pr[\max(X,Y) > x \mid \tau_1 \leq X \leq \tau_2, \tau_1 \leq Y \leq \tau_2] = 1 - [1 - \overline{F}_{X_{\tau_1,\tau_2}}(x)]^2, $$
	and
	$$ \textstyle \Pr[\min(X,Y) > x \mid \tau_1 \leq X \leq \tau_2, \tau_1 \leq Y \leq \tau_2] = [\overline{F}_{X_{\tau_1,\tau_2}}(x)]^2, $$
	we get that
	\begin{align*} 
			 \textstyle \Pr[\max(X,Y)>x|\tau_1 \leq X\leq \tau_2, \tau_1 \leq Y\leq \tau_2]&-Pr[\min(X,Y)>x|\tau_1 \leq X\leq \tau_2, \tau_1\leq Y \leq \tau_2] \\
			& \qquad = 2 \overline{F}_{X_{\tau_1,\tau_2}}(x)[1-\overline{F}_{X_{\tau_1,\tau_2}}(x)].
	\end{align*}
	Integrating both sides from $\tau_1$ to $\tau_2$ yields
	$$ \textstyle \mathbb{E}[|X - Y| \mid \tau_1 \leq X \leq \tau_2, \tau_1 \leq Y \leq \tau_2] = 2 \int_{\tau_1}^{\tau_2} \frac{\overline{F}_X(x) - \overline{F}_X(\tau_2)}{\overline{F}_X(\tau_1) - \overline{F}_X(\tau_2)} \left[ 1 - \frac{\overline{F}_X(x) - \overline{F}_X(\tau_2)}{\overline{F}_X(\tau_1) - \overline{F}_X(\tau_2)} \right] dx. $$
	Making use of the inequality $x(1 - x) \leq x |\ln x|, 0 < x < 1$, we obtain
	\begin{align*} 
		& \textstyle \mathbb{E}[|X - Y| \mid \tau_1 \leq X\leq \tau_2, \tau_1\leq Y \leq \tau_2] \\
		& \textstyle \qquad \leq 2 \int_{\tau_1}^{\tau_2} \frac{\overline{F}_X(x) - \overline{F}_X(\tau_2)}{\overline{F}_X(\tau_1) - \overline{F}_X(\tau_2)} \left| \ln \left( \frac{\overline{F}_X(x) - \overline{F}_X(\tau_2)}{\overline{F}_X(\tau_1) - \overline{F}_X(\tau_2)} \right) \right| dx \\
		& \textstyle \qquad = 2 H(X;\tau_1,\tau_2). 
	\end{align*}
	The left-side inequality follows as in Rao \cite{Rao2005} using Jensen's inequality.
\end{proof}

\begin{theorem} 
	Let $X$ and $Y$ be two non-negative absolutely continuous random variables with finite generalized failure rates $h_{X,1}(\tau_1,\tau_2)$ and $h_{Y,1}(\tau_1,\tau_2)$ and doubly truncated mean residual lifetimes $m_{X,1}(\tau_1,\tau_2)$ and $m_{Y,1}(\tau_1,\tau_2)$, respectively.
	
	\noindent (i) If for any $(\tau_1, \tau_2) \in D_X \cap D_Y$ it holds $h_{X,1}(\tau_1,\tau_2) \leq h_{Y,1}(\tau_1,\tau_2)$ and $m_{X,1}(x,\tau_2)$ is increasing in $x \geq 0$, then
	\begin{equation}
		H(Y;\tau_1,\tau_2) \leq H(X;\tau_1,\tau_2).
	\end{equation}
	
	\noindent (ii) If for any $(\tau_1, \tau_2) \in D_X \cap D_Y$ it holds $h_{Y,1}(\tau_1,\tau_2) \leq h_{X,1}(\tau_1,\tau_2)$ and $m_{X,1}(x,\tau_2)$ is decreasing in $x \geq 0$, then
	\begin{equation}
		H(X;\tau_1,\tau_2) \leq H(Y;\tau_1,\tau_2).
	\end{equation}
\end{theorem}

\begin{proof} (i) If $h_{X,1}(\tau_1,\tau_2) \leq h_{Y,1}(\tau_1,\tau_2)$, then from Lemma 1 (i) it follows that $m_{X,1}(x,\tau_2) \geq m_{Y,1}(\tau_1,\tau_2)$ and hence using Theorem 1 we get
	\begin{align}
		H(Y;\tau_1,\tau_2) & \textstyle = \frac{1}{\overline{F}_Y(\tau_1) - \overline{F}_Y(\tau_2)} \int_{\tau_1}^{\tau_2} m_{Y,1}(x,\tau_2) f_Y(x) dx \nonumber \\
		& \textstyle \leq \frac{1}{\overline{F}_Y(\tau_1) - \overline{F}_Y(\tau_2)} \int_{\tau_1}^{\tau_2} m_{X,1}(x,\tau_2) f_Y(x) dx \nonumber \\
		& \textstyle = \mathbb{E}[m_{X,1}(Y,\tau_2) \mid \tau_1 < Y \leq \tau_2],
	\end{align}
	where $f_Y(x)$ and $\overline{F}_Y(x)$ are the probability density function and the survival function of $Y$, respectively. By setting $\phi(x)=m_{X,1}(x,\tau_2)$, from Lemma 3, it follows that
	$$ \mathbb{E}[m_{X,1}(X,\tau_2) \mid \tau_1 \leq X \leq \tau_2] \geq \mathbb{E}[m_{X,1}(Y,\tau_2) \mid \tau_1 < Y \leq \tau_2], $$
	and hence from Theorem 1 we obtain
	\begin{equation}
		\mathbb{E}[m_{X,1}(Y,\tau_2) \mid \tau_1 < Y \leq \tau_2] \leq H(X;\tau_1,\tau_2).
	\end{equation}
	Now, the result in (38) follows immediately from (40) and (41).
	
	\noindent (ii) The proof is similar to (i) by interchanging the role of $X$ and $Y$ and reversing the inequalities.
\end{proof}

In the sequel, we give some applications of Theorem 8.

\begin{proposition}
	Let $X$ be a non-negative absolutely continuous random variable and $X_e$ be the equilibrium random variable corresponding to $X$. If $h_{X,1}(x,\tau_2)$ is increasing in $x$, then
	$$ H(X_e;\tau_1,\tau_2) \leq H(X;\tau_1,\tau_2). $$
\end{proposition}

\begin{proof}
	It holds
	\begin{align*}
		h_{X_e,1}(x,\tau_2) & = \frac{f_{X_e}(x)}{\overline{F}_{X_e}(x) - \overline{F}_{X_e}(\tau_2)} = \frac{\overline{F}_X(x)}{\int_x^{\tau_2} \overline{F}_X(t) dt} \\ 
		& = \frac{1}{m_X(x) - \overline{F}_X(\tau_2)/\overline{F}_X(x)}.
	\end{align*}
	
	If $h_{X,1}(x,\tau_2)$ is increasing in $x$ for any $\tau_2$, it follows that $h_X(x)$ is increasing in $x$, implying $m_X(x)$ is decreasing in $x$ (see \cite{Willmot2001}). Since $\overline{F}_X(x)$ is decreasing in $x$, we get $h_{X_e,1}(x,\tau_2)$ is also increasing in $x$. Hence, from Lemma 2 (ii) the doubly truncated mean residual lifetime $m_{X_e,1}(x,\tau_2)$ of $X_e$ is decreasing in $x$. Also, since $h_{X,1}(x,\tau_2)$ is increasing, from Lemma 5 it follows $h_{X,1}(x,\tau_2) \leq h_{X_e,1}(x,\tau_2)$. Now, the result follows from Theorem 8 (ii).
\end{proof}

Navarro et al. \cite{Navarro2010} compared the CREs of the random variables X and $X^{(p)}$, where the distribution of the random variable $X^{(p)}$ belongs to the \textit{proportional odds family} (POF), which is family of distributions arising in several scientific areas, such as in survival and reliability data. For a further discussion on POF models, one may refer to Kirmani and Gupta \cite{Kirmani2001}. If $g_p(x)$ and $\overline{G}_p(x)$ denote respectively the probability density function and the survival function of $X^{(p)}$, then
\begin{equation} \textstyle
	g_p(x) = \frac{p f_X(x)}{[1 - (1 - p) \overline{F}_X(x)]^2} \quad \text{and} \quad \overline{G}_p(x) = \frac{p \overline{F}_X(x)}{1 - (1 - p) \overline{F}_X(x)}, \quad 0 < p < 1.
\end{equation}
The parameter $p$ is called \textit{``tilt parameter''}. This is because the hazard rate of this family is shifted above the hazard rate of the underlying (baseline) distribution.

\begin{proposition} 
	Let $X$ be a non-negative absolutely continuous random variable and $X^{(p)}$ be a random variable with survival function given by (42). If $m_{X,1}(x,\tau_2)$ is increasing in $x$, then
	$$ H(X^{(p)};\tau_1,\tau_2) \leq H(X;\tau_1,\tau_2). $$
\end{proposition}

\begin{proof}
	Since $\overline{F}_X(x) \geq \overline{F}_X(\tau_2)$ for $x \leq \tau_2$ and $0 < p < 1$, it can be verified that $h_{X,1}(x,\tau_2) \leq h_{X^{(p)},1}(x,\tau_2)$, and hence the result follows immediately from Theorem 8 (i).
\end{proof}

\section{Monotonicity results and further bounds for $\boldsymbol{H(X;\tau_1,\tau_2)}$} \label{sec5}

To obtain the next result, we assume that for any fixed value of $\tau_2$, the derivative of $H(X;\tau_1,\tau_2)$ with respect to $\tau_1$ exists. First, based on $H(X;\tau_1,\tau_2)$, we define two new classes of distributions.

\begin{definition}
	A random variable $X$ is said to have an increasing (decreasing) cumulative residual interval entropy, or ICRIE (DCRIE), if $H(X;\tau_1,\tau_2)$ is increasing (decreasing) in $\tau_1$ for any fixed value of $\tau_2$.
\end{definition}

\begin{theorem}
	Let $X$ be a non-negative absolutely continuous random variable. The random variable $X$ is ICRIE (DCRIE), if and only if for all $(\tau_1,\tau_2)\in D$
	$$ \textstyle H(X;\tau_1,\tau_2) \geq (\leq) \, m_{X,1}(\tau_1,\tau_2). $$ 
\end{theorem}

\begin{proof}
	For any fixed value of $\tau_2$ and using (8), the differentiation of both sides of (40) gives
	\begin{align*}
		\textstyle \frac{\partial H(X;\tau_1,\tau_2)}{\partial \tau_1} & = \textstyle \ln[\overline{F}_X(\tau_1) - \overline{F}_X(\tau_2)]\left\{1 + \frac{\partial m_{X,1}(\tau_1,\tau_2)}{\partial \tau_1}\right\} - m_{X,1}(\tau_1,\tau_2) h_{X,1}(\tau_1,\tau_2) \\
		& \quad + h_{X,1}(\tau_1,\tau_2)\left\{H(X;\tau_1,\tau_2) - m_{X,1}(\tau_1,\tau_2)\ln[\overline{F}_X(\tau_1) - \overline{F}_X(\tau_2)]\right\}.
	\end{align*}
	Hence, using (9) it follows that
	\begin{equation}
		\textstyle \frac{\partial H(X;\tau_1,\tau_2)}{\partial \tau_1} = h_{X,1}(\tau_1,\tau_2)\left\{H(X;\tau_1,\tau_2) - m_{X,1}(\tau_1,\tau_2)\right\}.
	\end{equation}
	Now, the result follows directly from (43).
\end{proof}

\begin{corollary}
	Let $X$ be a non-negative absolutely continuous random variable. If for any fixed value of $\tau_2$, $m_{X,1}(\tau_1,\tau_2)$ is increasing (decreasing) in $\tau_1$, then the random variable $X$ is ICRIE (DCRIE). 
\end{corollary} 

\begin{proof}
	From Theorem 1 we have
	$$ \textstyle H(X;\tau_1,\tau_2) = \frac{1}{\overline{F}_X(\tau_1) - \overline{F}_X(\tau_2)} \int_{\tau_1}^{\tau_2} m_{X,1}(x,\tau_2) f_X(x)\, dx. $$
	If $m_{X,1}(\tau_1,\tau_2)$ is increasing (decreasing) in $\tau_1$, then for $x \geq \tau_1$ it holds $m_{X,1}(x,\tau_2) \geq (\leq) m_{X,1}(\tau_1,\tau_2)$, and thus we obtain
	$$ \textstyle H(X;\tau_1,\tau_2) \geq (\leq) \frac{m_{X,1}(\tau_1,\tau_2)}{\overline{F}_X(\tau_1) - \overline{F}_X(\tau_2)} \int_{\tau_1}^{\tau_2} f_X(x)\, dx, $$
	i.e.,
	$$ \textstyle H(X;\tau_1,\tau_2) \geq (\leq) m_{X,1}(\tau_1,\tau_2). $$
	Now, the result follows from Theorem 9.
\end{proof}

\begin{example}
	Using Corollary 7, to illustrate the behavior of the mean residual lifetime function $ m_{X,1}(\tau_1, \tau_2) $ and the associated entropy measure $ H(X; \tau_1, \tau_2) $, we present graphical representations for specific distributional cases. These visualizations help us to understand the monotonicity properties of $ H(X; \tau_1, \tau_2) $ with respect to $ \tau_1 $ under varying parameter settings. In Figure~\ref{fig:5.1}, we consider the beta distribution with different shape parameters, while in Figure~\ref{fig:5.2}, we explore the exponential distribution with various rate parameters. Each figure includes plots of both the mean residual lifetime function and the entropy measure for a fixed value of $ \tau_2 $, showing how these quantities evolve as $ \tau_1 $ increases. We observe that for both distributions, the doubly truncated mean residual lifetime $ m_{X,1}(\tau_1, \tau_2) $ is a decreasing function in $ \tau_1 $, and the same is also true for $ H(X; \tau_1, \tau_2) $, as expected by Corollary 7. 
	
	\begin{figure}[H]
			\includegraphics[width=.95\linewidth]{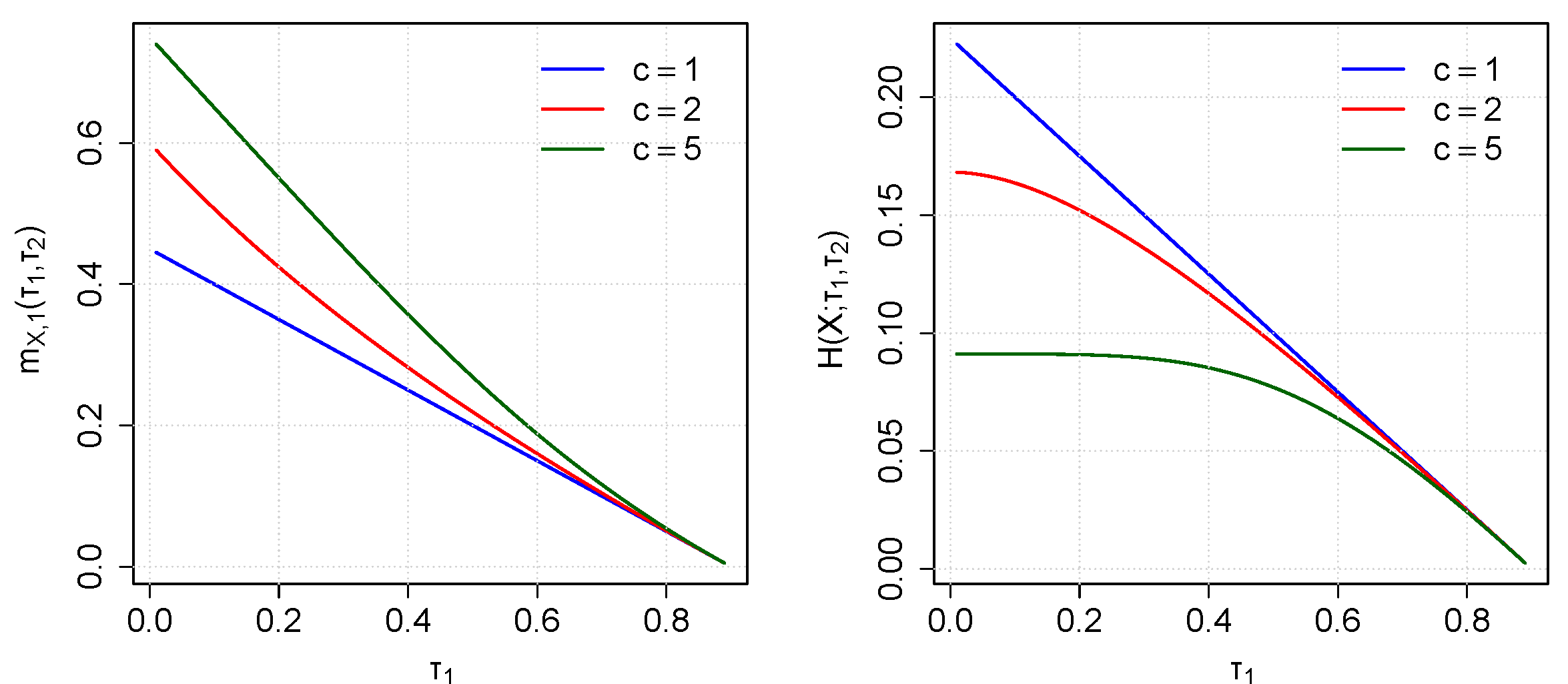}
			\caption{Plots of the mean residual lifetime function $ m_{X,1}(\tau_1,\tau_2) $ (left panel) and the entropy measure $ H(X; \tau_1, \tau_2) $ (right panel) versus $\tau_1 $ with fixed $ \tau_2 = 0.9 $ for the beta distribution with $\overline{F}(x) = 1 - x^c$ for $0 < x < 1$, where $ c = 1, 2, 5 $.}
			\label{fig:5.1}
	\end{figure}

	\begin{figure}[H]
			\includegraphics[width=.95\linewidth]{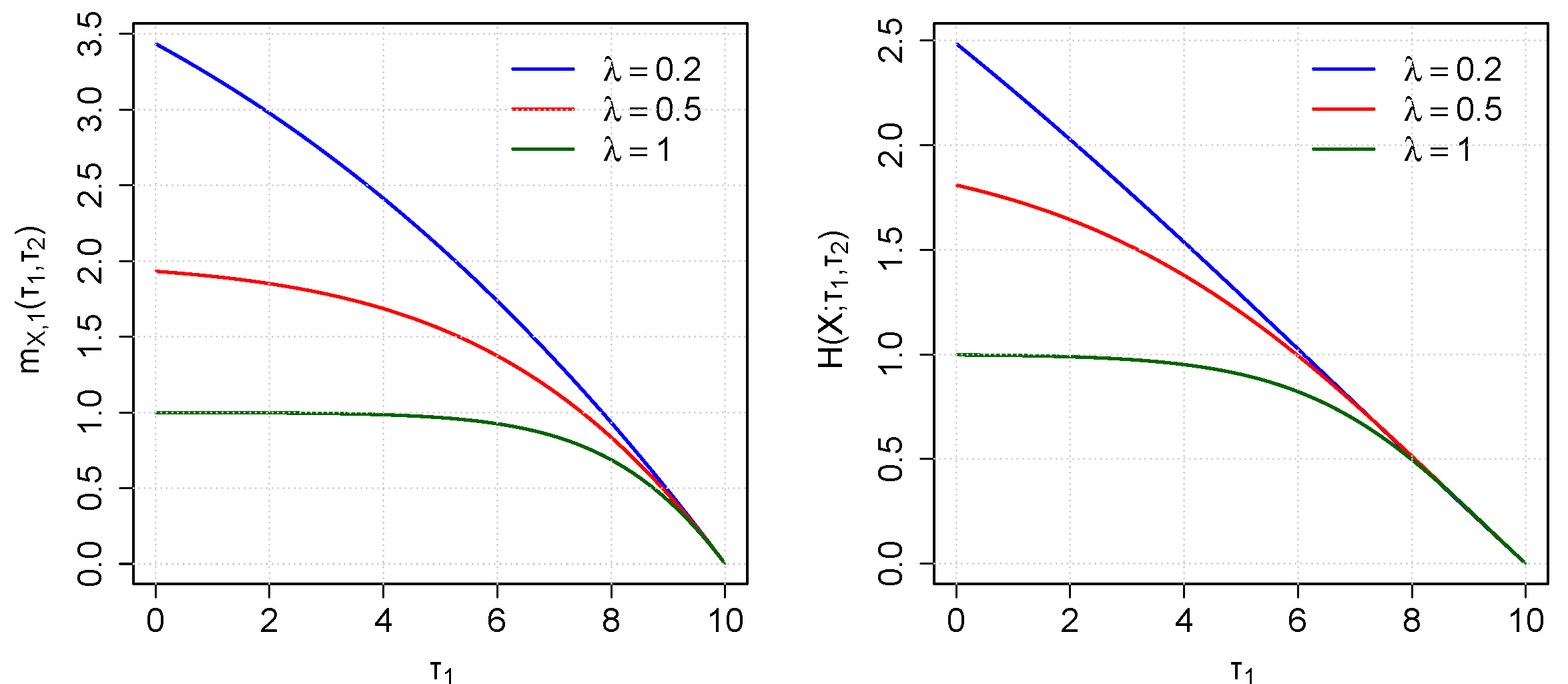}
			\caption{Plots of the mean residual lifetime function $ m_{X,1}(\tau_1,\tau_2) $ (left panel) and the entropy measure $ H(X; \tau_1, \tau_2) $ (right panel) versus $ \tau_1 $ with fixed $ \tau_2 = 10 $ for the exponential distribution with $\overline{F}(x) = \exp(-\lambda x)$ for $x \geq 0$, where $ \lambda = 0.2, 0.5, 1 $.}
			\label{fig:5.2}
	\end{figure}
\end{example}

We know from Lemma 2 (ii) that if $h_{X,1}(x,y)$ is decreasing (increasing) in $x$, then $m_{X,1}(x,y)$ is increasing (decreasing) in $x$. Therefore, the condition that $h_{X,1}(x,y)$ is decreasing (increasing) in $x$ is stronger than the condition that $m_{X,1}(x,y)$ is increasing (decreasing) in $x$. In the following proposition, based on the relevation transform representation of $H(X;\tau_1,\tau_2)$ given by Theorem 3, we prove that Corollary 7 also holds under the stronger condition that $h_{X,1}(x,y)$ is decreasing (increasing) in $x$ and, furthermore, is given a two-sided bound for $H(X;\tau_1,\tau_2)$.

\begin{proposition} 
	Let $X$ be a non-negative absolutely continuous random variable. If $h_{X,1}(x,y)$ is decreasing (increasing) in $x$, then, for all $(\tau_1,\tau_2)\in D_X$
	$$ \textstyle m_{X,1}(\tau_1,\tau_2) \leq (\geq) H(X;\tau_1,\tau_2) \leq (\geq) \frac{1}{2} h_{X,1}(\tau_1,\tau_2)\mathbb{E}[(X - \tau_1)^2 \mid \tau_1 \leq X \leq \tau_2], $$
	and thus $X$ is ICRIE (DCRIE).
\end{proposition}

\begin{proof}
	If $h_{X,1}(x,\tau_2)$ is decreasing (increasing) in $x$, then, for $\tau_1 \leq u \leq x$ it holds
	$$ \textstyle (x - \tau_1)h_{X,1}(x,\tau_2) \leq (\geq) \int_{\tau_1}^x h_{X,1}(u,\tau_2)\, du \leq (\geq) (x - \tau_1)h_{X,1}(\tau_1,\tau_2). $$
	Then, from (25) it follows that
	\begin{align*} 
		& \textstyle \overline{F}_{X_{\tau_1,\tau_2}}(x) + \overline{F}_{X_{\tau_1,\tau_2}}(x)(x - \tau_1)h_{X,1}(x,\tau_2) \leq (\geq) (\overline{F}_{X_{\tau_1,\tau_2}} \# \overline{F}_{X_{\tau_1,\tau_2}})(x) \\
		& \textstyle \qquad \qquad \leq (\geq) \overline{F}_{X_{\tau_1,\tau_2}}(x) + \overline{F}_{X_{\tau_1,\tau_2}}(x)(x - \tau_1)h_{X,1}(\tau_1,\tau_2).
	\end{align*}
	Integrating all sides from $\tau_1$ to $\tau_2$ and using (8) we obtain
	\begin{align*}
		& \textstyle m_{X,1}(\tau_1,\tau_2) + \int_{\tau_1}^{\tau_2} \overline{F}_{X_{\tau_1,\tau_2}}(x)(x - \tau_1)h_{X,1}(x,\tau_2)\, dx \\
		& \textstyle \qquad \qquad \leq (\geq) \textstyle  \int_{\tau_1}^{\tau_2} (\overline{F}_{X_{\tau_1,\tau_2}} \# \overline{F}_{X_{\tau_1,\tau_2}})(x)\, dx \\
		& \textstyle \qquad \qquad \leq (\geq) m_{X,1}(\tau_1,\tau_2) + \int_{\tau_1}^{\tau_2} \overline{F}_{X_{\tau_1,\tau_2}}(x)(x - \tau_1)h_{X,1}(\tau_1,\tau_2)\, dx.
	\end{align*}
	Therefore, from Theorem 3 the above relationship yields
	\begin{align*}
		\textstyle \int_{\tau_1}^{\tau_2} \overline{F}_{X_{\tau_1,\tau_2}}(x)(x - \tau_1)h_{X,1}(x,\tau_2)\, dx & \leq (\geq) H(X;\tau_1,\tau_2) \\
		& \textstyle \leq (\geq) \int_{\tau_1}^{\tau_2} \overline{F}_{X_{\tau_1,\tau_2}}(x)(x - \tau_1)h_{X,1}(\tau_1,\tau_2)\, dx, 
	\end{align*}
	or equivalently,
	\begin{align*} 
		& \textstyle \frac{1}{\overline{F}_X(\tau_1) - \overline{F}_X(\tau_2)} \int_{\tau_1}^{\tau_2} (x - \tau_1)f_X(x)\, dx \\ 
		& \qquad \qquad \textstyle \leq (\geq) H(X;\tau_1,\tau_2) \leq (\geq) \frac{h_{X,1}(\tau_1,\tau_2)}{\overline{F}_X(\tau_1) - \overline{F}_X(\tau_2)} \int_{\tau_1}^{\tau_2} (x - \tau_1)[\overline{F}_X(x) - \overline{F}_X(\tau_2)]\, dx, 
	\end{align*}
	i.e.,
	$$ \textstyle m_{X,1}(\tau_1,\tau_2) \leq (\geq) H(X;\tau_1,\tau_2) \leq (\geq) \frac{1}{2} h_{X,1}(\tau_1,\tau_2) \mathbb{E}[(X - \tau_1)^2 \mid \tau_1 \leq X \leq \tau_2]. $$
	Now, since
	$$ \textstyle \mathbb{E}[(X - \tau_1)^2 \mid \tau_1 \leq X \leq \tau_2] = \frac{2}{\overline{F}_X(\tau_1) - \overline{F}_X(\tau_2)} \int_{\tau_1}^{\tau_2} (x - \tau_1)[\overline{F}_X(x) - \overline{F}_X(\tau_2)]\, dx, $$
	the result follows from the above two-sided inequalities.
\end{proof}

By setting $\tau_2 \to \infty$, from Proposition 10 we immediately obtain the following corollary.

\begin{corollary} 
	Let $X$ be a non-negative absolutely continuous random variable. If $X$ is DFR (IFR) then
	\begin{equation}
		\textstyle m_X(t) \leq (\geq) \mathcal{E}(X;t) \leq (\geq) \frac{1}{2} h_X(t) \mathbb{E}[(X - t)^2 \mid X > t],
	\end{equation}
	and
	\begin{equation}
		\textstyle \mathbb{E}(X) \leq (\geq) \mathcal{E}(X) \leq (\geq) \frac{1}{2} f_X(0) \mathbb{E}[X^2].
	\end{equation}
\end{corollary}

\begin{remark}
	(i) Asadi and Zohrevand \cite{Asadi2007} proved in Corollary 4.4 that if $X$ is IMRL (DMRL) then 
	$$ \textstyle \mathcal{E}(X;t) \geq (\leq) m_X(t). $$
	Since the DFR (IFR) class is a subclass of the IMRL (DMRL) class, from Corollary 8 it follows that Corollary 4.4 of \cite{Asadi2007} is still valid under the stronger condition that $X$ is DFR (IFR).  
	(ii) We consider that $X$ has an exponential distribution with parameter $\lambda > 0$. Since it has a constant hazard rate, it follows that $X$ is DFR and IFR. In this case (44) and (45) must hold as equalities, i.e., we must have
	\begin{equation}
		\textstyle \mathcal{E}(X;t) = m_X(t) = \frac{1}{2} h_X(t) \mathbb{E}[(X - t)^2 \mid X > t],
	\end{equation}
	and
	\begin{equation}
		\textstyle \mathcal{E}(X) = \mathbb{E}(X) = \frac{1}{2} f_X(0) \mathbb{E}[X^2].
	\end{equation}
	Since
	$$ \textstyle \mathcal{E}(X;t) = \mathcal{E}(X) = \frac{1}{\lambda}, \quad h_X(t) = f_X(0) = \lambda, \quad m_X(t) = \mathbb{E}(X) = \frac{1}{\lambda}, $$
	and
	$$ \textstyle \mathbb{E}[(X - t)^2 \mid X > t] = \mathbb{E}[X^2] = \frac{2}{\lambda^2}, $$
	it follows that (46) and (47) are valid. Therefore, the bounds in (44) and (45) are sharp.
\end{remark}

The following corollary considers the behavior of the ICRIE (DCRIE) class under increasing linear transformations.

\begin{corollary} 
	Let $X$ be a non-negative absolutely continuous random variable, $a > 0$ and $b \geq 0$ be constants. If $X$ is ICRIE (DCRIE), then $aX + b$ is ICRIE (DCRIE).
\end{corollary}

\begin{proof} 
	Under the conditions of theorem, from Theorem 1 it holds that
	\begin{equation}
		\textstyle a H\left(X;\frac{\tau_1 - b}{a}, \frac{\tau_2 - b}{a} \right) \geq (\leq) a m_{X,1}\left(\frac{\tau_1 - b}{a}, \frac{\tau_2 - b}{a} \right).
	\end{equation}
	Since it can be easily shown that
	$$ \textstyle m_{aX+b,1}(\tau_1, \tau_2) = a m_{X,1}\left(\frac{\tau_1 - b}{a}, \frac{\tau_2 - b}{a} \right), $$
	then from Proposition 1 and (48) we get that 
	$$ \textstyle H(aX+b ;\tau_1,\tau_2 ) \leq (\geq) m_{aX+b,1}(\tau_1,\tau_2), $$
	and hence the result follows from Theorem 9.
\end{proof}

From Corollary 9, we conclude that the ICRIE (DCRIE) class is closed under increasing linear transformations. Hence, it is interesting to examine whether this also holds for any increasing transformation. In Theorem 10, it is shown a more general result for the DCRIE but under an increasing convex transformation. More precisely, it is shown that the DCRIE class is closed under increasing convex transformations. For this, first we need the following proposition in which we give an equivalent condition to that of Theorem 9 in order for the random variable $X$ to be ICRIE (DCRIE).

\begin{proposition}
	Let $X$ be a non-negative absolutely continuous random variable. The random variable $X$ is ICRIE (DCRIE), if and only if for all $(\tau_1,\tau_2)\in D_X$
	\begin{equation}
		\textstyle \int_{\tau_1}^{\tau_2} \frac{\overline{F}_X(x) - \overline{F}_X(\tau_2)}{\overline{F}_X(\tau_1) - \overline{F}_X(\tau_2)} \left\{ \ln \left( \frac{\overline{F}_X(x) - \overline{F}_X(\tau_2)}{\overline{F}_X(\tau_1) - \overline{F}_X(\tau_2)} \right) + 1 \right\} dx \leq (\geq) 0.
	\end{equation}
\end{proposition}

\begin{proof} Under the conditions of theorem, from Theorem 9 and (8) we get
	\begin{align*}
		\textstyle - \int_{\tau_1}^{\tau_2} \frac{\overline{F}_X(x) - \overline{F}_X(\tau_2)}{\overline{F}_X(\tau_1) - \overline{F}_X(\tau_2)} \ln \left( \frac{\overline{F}_X(x) - \overline{F}_X(\tau_2)}{\overline{F}_X(\tau_1) - \overline{F}_X(\tau_2)} \right) dx 
		\geq (\leq)
		\int_{\tau_1}^{\tau_2} \frac{\overline{F}_X(x) - \overline{F}_X(\tau_2)}{\overline{F}_X(\tau_1) - \overline{F}_X(\tau_2)} dx,
	\end{align*}
	from which (49) follows immediately. 
\end{proof}

\begin{theorem}
	Let $X$ be a non-negative absolutely continuous random variable, and $\varphi$ be an increasing convex function on $[0,+\infty)$ such that the derivative $\varphi'$ of $\varphi$ is continuous. If $X$ is DCRIE, then $\varphi(X)$ is also DCRIE.
\end{theorem}

\begin{proof} Suppose that $X$ is DCRIE. Then, from (49) it follows that it also holds
	\begin{equation*}
			\textstyle \int_{\varphi^{-1}(\tau_1)}^{\varphi^{-1}(\tau_2)} 
			\frac{\overline{F}_X(x) - \overline{F}_X(\varphi^{-1}(\tau_2))}{\overline{F}_X(\varphi^{-1}(\tau_1)) - \overline{F}_X(\varphi^{-1}(\tau_2))}
			\left\{\ln \left(\frac{\overline{F}_X(x) - \overline{F}_X(\varphi^{-1}(\tau_2))}{\overline{F}_X(\varphi^{-1}(\tau_1)) - \overline{F}_X(\varphi^{-1}(\tau_2))}\right) + 1
			\right\} dx \geq 0.
	\end{equation*}
	Since $\varphi(x)$ is increasing convex function implying that $\varphi'(x)$ is non-negative increasing, from the above relationship and Corollary 7.2 (p. 121) of Barlow and Proschan \cite{Barlow1975} we have that
	\begin{equation}
			\textstyle \int_{\varphi^{-1}(\tau_1)}^{\varphi^{-1}(\tau_2)} 
			\varphi'(x) \frac{\overline{F}_X(x) - \overline{F}_X(\varphi^{-1}(\tau_2))}{\overline{F}_X(\varphi^{-1}(\tau_1)) - \overline{F}_X(\varphi^{-1}(\tau_2))}
			\left\{
			\ln \left(
			\frac{\overline{F}_X(x) - \overline{F}_X(\varphi^{-1}(\tau_2))}{\overline{F}_X(\varphi^{-1}(\tau_1)) - \overline{F}_X(\varphi^{-1}(\tau_2))}
			\right) + 1
			\right\} dx \geq 0.
	\end{equation}
	By setting $Y = \varphi(X)$, we obtain
	\begin{align*}
		& \textstyle \quad \int_{\tau_1}^{\tau_2} \frac{\overline{F}_Y(y) - \overline{F}_Y(\tau_2)}{\overline{F}_Y(\tau_1) - \overline{F}_Y(\tau_2)} 
		\left\{ \ln\left( \frac{\overline{F}_Y(y) - \overline{F}_Y(\tau_2)}{\overline{F}_Y(\tau_1) - \overline{F}_Y(\tau_2)} \right) + 1 \right\} dy \\
		& \textstyle \qquad \qquad = \int_{\varphi^{-1}(\tau_1)}^{\varphi^{-1}(\tau_2)} \frac{\overline{F}_X(\varphi^{-1}(y)) - \overline{F}_X(\varphi^{-1}(\tau_2))}{\overline{F}_X(\varphi^{-1}(\tau_1)) - \overline{F}_X(\varphi^{-1}(\tau_2))} 
		\left\{ \ln\left( \frac{\overline{F}_X(\varphi^{-1}(y)) - \overline{F}_X(\varphi^{-1}(\tau_2))}{\overline{F}_X(\varphi^{-1}(\tau_1)) - \overline{F}_X(\varphi^{-1}(\tau_2))} \right) + 1 \right\} dy \\
		& \textstyle \qquad \qquad = \int_{\varphi^{-1}(\tau_1)}^{\varphi^{-1}(\tau_2)} 
		\frac{\varphi'(x)\left( \overline{F}_X(x) - \overline{F}_X(\varphi^{-1}(\tau_2)) \right)}{\overline{F}_X(\varphi^{-1}(\tau_1)) - \overline{F}_X(\varphi^{-1}(\tau_2))} 
		\left\{ \ln\left( \frac{\overline{F}_X(x) - \overline{F}_X(\varphi^{-1}(\tau_2))}{\overline{F}_X(\varphi^{-1}(\tau_1)) - \overline{F}_X(\varphi^{-1}(\tau_2))} \right) + 1 \right\} dx \\
		& \textstyle \qquad \qquad \geq 0,
	\end{align*}
	where the last inequality follows from (50). Therefore, from (49) we get that the random variable $Y$ is DCRIE, which completes the proof.
\end{proof}

\begin{theorem}
	Let $X$ be a non-negative absolutely continuous random variable. If the random variable $X$ is ICRIE (DCRIE), then, for all $(\tau_1,\tau_2) \in D_X$
	\begin{equation}
		H(X; \tau_1, \tau_2) \leq (\geq) \frac{1}{h_{X,1}(\tau_1, \tau_2)}.
	\end{equation}
\end{theorem}

\begin{proof}
	We have
	\begin{align*}
			& \textstyle \lim_{\tau_1 \to \tau_2} H(X;\tau_1,\tau_2) \\
			& \quad \textstyle = \lim_{\tau_1 \to \tau_2} \frac{- \int_{\tau_1}^{\tau_2} [\overline{F}_X(x) - \overline{F}_X(\tau_2)] \ln[\overline{F}_X(x) - \overline{F}_X(\tau_2)] dx + \ln[\overline{F}_X(x) - \overline{F}_X(\tau_2)] \int_{\tau_1}^{\tau_2} [\overline{F}_X(x) - \overline{F}_X(\tau_2)] dx}{\overline{F}_X(\tau_1) - \overline{F}_X(\tau_2)} \\
			& \quad \textstyle = \lim_{\tau_1 \to \tau_2} \frac{\int_{\tau_1}^{\tau_2} [\overline{F}_X(x) - \overline{F}_X(\tau_2)] dx}{\overline{F}_X(\tau_1) - \overline{F}_X(\tau_2)} \\
			& \quad \textstyle = \frac{1}{h_{X,1}(\tau_1, \tau_2)}.
	\end{align*}
	If $H(X;\tau_1,\tau_2)$ is increasing (decreasing) in $\tau_1$, then, for all $\tau_1 \leq \tau_2$ it holds $$ \textstyle H(X;\tau_1,\tau_2) \leq (\geq) H(X;\tau_2,\tau_2) = \lim_{\tau_1 \to \tau_2} H(X;\tau_1,\tau_2), $$ and hence (51) follows from the last two relationships.
\end{proof}

Using Theorems 9 and 11, a two-sided bound for $H(X;\tau_1,\tau_2)$ can be immediately obtained if $H(X;\tau_1,\tau_2)$ increases in $\tau_1$ for any fixed value of $\tau_2$. Thus, we have the following corollary.

\begin{corollary}
	Let $X$ be a non-negative absolutely continuous random variable. If the random variable $X$ is ICRIE (DCRIE), then, for all $(\tau_1,\tau_2) \in D$
	$$ \textstyle m_{X,1}(\tau_1, \tau_2) \leq (\geq) H(X;\tau_1,\tau_2) \leq (\geq) \frac{1}{h_{X,1}(\tau_1, \tau_2)}. $$
\end{corollary}

\end{document}